\documentclass{article}
\usepackage{amsmath,amssymb,amsthm,color,graphicx,diagbox,tikz,colortbl,array}
\usepackage[margin=1.2in]{geometry}

\title{Doubly partially conservative sentences}
\author{Haruka Kogure\footnote{Email: kogure@stu.kobe-u.ac.jp}
\footnote{Graduate School of System Informatics, Kobe University, 1-1 Rokkodai, Nada, Kobe 657-8501, Japan.}
and Taishi Kurahashi\footnote{Email: kurahashi@people.kobe-u.ac.jp}
\footnote{Graduate School of System Informatics, Kobe University, 1-1 Rokkodai, Nada, Kobe 657-8501, Japan.}}
\date{}

\theoremstyle{plain}
\newtheorem{thm}{Theorem}[section]
\newtheorem*{thm*}{Theorem}
\newtheorem{lem}[thm]{Lemma}
\newtheorem{prop}[thm]{Proposition}
\newtheorem{cor}[thm]{Corollary}
\newtheorem{fact}[thm]{Fact}
\newtheorem*{fact*}{Fact}

\newtheorem*{prob*}{Problem}
\newtheorem*{cl*}{Claim}
\newtheorem{cl}{Claim}
\newtheorem*{scl*}{Subclaim}

\theoremstyle{definition}
\newtheorem{defn}[thm]{Definition}

\newcommand{\PA}{\mathsf{PA}}
\newcommand{\PR}{\mathrm{Pr}}
\newcommand{\Prf}{\mathrm{Prf}}

\newcommand{\Con}{\mathrm{Con}}

\newcommand{\RFN}{\mathrm{RFN}}
\newcommand{\gn}[1]{\ulcorner#1\urcorner}

\newcommand{\num}{\overline}
\newcommand{\True}{\mathrm{True}}

\newcommand{\SP}{\Sigma_n {\downarrow} \Pi_n}
\newcommand{\Cons}{\mathrm{Cons}}
\newcommand{\HCons}{\mathrm{HCons}}
\newcommand{\DCons}{\mathrm{DCons}}
\newcommand{\HDCons}{\mathrm{HDCons}}

\newcommand{\Th}{\mathrm{Th}}

\begin{document}

\maketitle

\begin{abstract}
The purpose of the present paper is to analyze several variants of Solovay's theorem on the existence of doubly partially conservative sentences. 
First, we investigate $\Theta$ sentences that are doubly $(\Gamma, \Lambda)$-conservative over $T$ for several triples $(\Theta, \Gamma, \Lambda)$. 
Among other things, we prove that the existence of a $\Delta_{n+1}(\PA)$ sentence that is doubly $(\Sigma_n, \Sigma_n)$-conservative over $T$ is equivalent to the $\Sigma_{n+1}$-inconsistency of $T$ over $\PA$. 
Secondly, we study $\Theta$ sentences that are hereditarily doubly $(\Gamma, \Lambda)$-conservative over $T$ for several triples $(\Theta, \Gamma, \Lambda)$. 
\end{abstract}

\maketitle

\section{Introduction}

Partially conservative sentences have been studied in the context of the incompleteness theorems (cf.~\cite{HP,Lin}), and also recently have been studied by the authors \cite{KOSV,KK}. 
We fix $n \geq 1$ and a consistent c.e.~extension $T$ of Peano Arithmetic $\PA$. 
For a class $\Gamma$ of sentences, a sentence $\varphi$ is said to be $\Gamma$-conservative over a theory $T$ if every $\Gamma$ sentence provable in $T + \varphi$ is already provable in $T$. 
Guaspari \cite[Theorem 2.4]{Gua} proved that there exists a $\Sigma_n$ (resp.~$\Pi_n$) sentence $\varphi$ which is $\Pi_n$-conservative (resp.~$\Sigma_n$-conservative) over $T$ and $T \nvdash \varphi$. 
It is easy to see that a sentence $\varphi$ is $\Pi_n$-conservative (resp.~$\Sigma_n$-conservative) over $T$ if and only if for any set $X$ of $\Sigma_n$ sentences (resp.~$\Pi_n$ sentences), $T + X \nvdash \neg \varphi$ whenever $T + X$ is consistent. 
So, Guaspari's sentences are independent of $T$ and are irrefutable in this strong sense. 

In this line of thought, it is natural to take into account sentences that are simultaneously unprovable and irrefutable in the strong sense, and the existence of such sentences were proved by Solovay (cf.~\cite[Theorem 2.7]{Gua}, \cite[Theorem 3.4.(3)]{Haj87}, \cite[Theorem 5.3]{Lin}, and \cite[Application 2]{Smo81}). 
Namely, there exists a $\Sigma_n$ sentence $\varphi$ such that $\varphi$ is $\Pi_n$-conservative over $T$ and $\neg \varphi$ is $\Sigma_n$-conservative over $T$. 
Solovay's theorem has been interestingly strengthened by Arana \cite{Ara}, Shavrukov \cite{Sha}, and Lindstr\"om \cite{Lin11} in the form of a chain of sentences.

Several variants of the Solovay-type theorem have also been studied. 
H\'ajek \cite[Theorem 2]{Haj79} proved that there exists a $\Delta_{n+1}(\PA)$ sentence $\varphi$ such that $\varphi$ is $\mathcal{B}(\Sigma_n)$-conservative over $T$ and $\neg \varphi$ is $\Pi_n$-conservative over $T$, where $\mathcal{B}(\Sigma_n)$ is the class of all Boolean combinations of $\Sigma_n$ sentences.  
We call such a sentence $\varphi$ a \textit{doubly $(\mathcal{B}(\Sigma_n), \Pi_n)$-conservative sentence}. 
Doubly $(\Pi_1, \Pi_1)$-conservative sentences are known as Orey sentences in the context of interpretability \cite{Ore} (see also \cite{Lin}). 
Recently, the second author and Visser \cite[Theorem 3.9]{KV} proved the existence of a $\Pi_{n+1}$ consistency statement that is doubly $(\Pi_n, \Pi_n)$-conservative. 
On the other hand, for some classes $\Gamma$ and $\Lambda$ of sentences, the situation where doubly $(\Gamma, \Lambda)$-conservative sentences do not exist is discussed in Lindstr\"om's textbook \cite[Exercise 5.6]{Lin} (see Fact \ref{Fact} below).

Guaspari \cite[Theorem 2.4]{Gua} actually proved the existence of sentences enjoying a stronger hereditary conservation property.
That is, it is proved that there exists a $\Sigma_n$ (resp.~$\Pi_n$) sentence $\varphi$ such that for any theory $U$ with $T \vdash U \vdash \PA$, $\varphi$ is $\Pi_n$-conservative (resp.~$\Sigma_n$-conservative) over $U$ and $T \nvdash \varphi$. 
Smory\'nski \cite[p.~343]{Smo81} mentioned that the Solovay-type sentences with this subtheory property had not been analyzed. 
Lindstr\"om \cite[Exercise 5.5,(c)]{Lin} described that hereditarily doubly $(\Pi_n, \Sigma_n)$-conservative $\Sigma_n$ sentences do not always exist.
Recently, the authors introduced the notion of $\SP$-conservativity (see Definition \ref{cons} below) and proved that $T$ is $\SP$-conservative over $\PA$ if and only if there exists a $\Sigma_n$ sentence that is hereditarily doubly $(\Pi_n, \Sigma_n)$-conservative over $T$ \cite[Theorem 5.7]{KK}. 

The purpose of the present paper is to further analyze several variants of the Solovay-type theorem and its hereditary version in order to clearly understand the situation surrounding doubly partially conservative sentences. 
In Section \ref{Sec:DCS}, we study $\Theta$ sentences that are doubly $(\Gamma, \Lambda)$-conservative for several triples $(\Theta, \Gamma, \Lambda)$. 
Tables \ref{table0}, \ref{table1}, and \ref{table2} below summarize the situations of the existence of such a sentence for $\Theta = \Sigma_n \land \Pi_n$, $\mathcal{B}(\Sigma_n)$, and $\Delta_{n+1}(\PA)$, respectively. 
In particular, Theorem \ref{Consis_Cons} that is a strengthening of Lindstr\"om's observation \cite[Exercise 5.6.(c)]{Lin} may be interesting: We prove that there exists a $\Delta_{n+1}(\PA)$ sentence that is doubly $(\Sigma_n, \Sigma_n)$-conservative over $T$ if and only if $T$ is not $\Sigma_{n+1}$-consistent over $\PA$. 
Here, $\Sigma_{n+1}$-consistency over $\PA$ is a newly introduced relativized version of Kreisel's $n$-consistency \cite{Kre57} that is a stratification of G\"odel's $\omega$-consistency. 
In Section \ref{Sec:HDCS}, we investigate $\Theta$ sentences that are hereditarily doubly $(\Gamma, \Lambda)$-conservative. 
Tables \ref{table3}, \ref{table4}, \ref{table5}, \ref{table6}, and \ref{table7} summarize the situations of the existence of such a sentence for $\Theta = \Sigma_n$, $\Sigma_n \land \Pi_n$, $\mathcal{B}(\Sigma_n)$, $\Delta_{n+1}(\PA)$, and $\Sigma_{n+1}$, respectively.

\section{Preliminaries}\label{Sec:Pre}

Throughout the present paper, we may assume that $T$ always denotes a consistent c.e.~extensions $T$ of Peano Arithmetic $\PA$ in the language $\mathcal{L}_A$ of first-order arithmetic. 
We also assume that $n$ denotes a natural number with $n \geq 1$. 
Let $\omega$ denote the set of all natural numbers. 
Let $\mathbb{N}$ denote the $\mathcal{L}_A$-structure that is the standard model of first-order arithmetic, whose domain is $\omega$. 
For each $k \in \omega$, let $\num{k}$ denote the numeral of $k$. 
For each $\mathcal{L}_A$-formula $\varphi$, let $\gn{\varphi}$ denote the numeral of the fixed G\"{o}del number of $\varphi$.
A formula is called $\Sigma_n \wedge \Pi_n$ (resp.~$\Sigma_n \lor \Pi_n$) iff it is of the form $\sigma \wedge \pi$ (resp.~$\sigma \lor \pi$) for some $\Sigma_n$ formula $\sigma$ and $\Pi_n$ formula $\pi$. 
A formula is called $\mathcal{B}(\Sigma_n)$ iff it is a Boolean combination of $\Sigma_n$ formulas.
For a theory $U$, a formula is called $\Delta_n(U)$ iff it is $U$-provably equivalent to both some $\Sigma_n$ formula and some $\Pi_n$ formula. 
We assume that the set of all $\Gamma$ sentences is denoted $\Gamma$. 

Here, we introduce the useful witness comparison notation. 
For formulas $\varphi \equiv \exists x \alpha (x)$ and $\psi \equiv \exists x \beta(x)$, we define the formulas $\varphi \prec \psi$ and $\varphi \preccurlyeq \psi$ as follows: 
\begin{itemize}
\item 
$\varphi \prec \psi : \equiv \exists x \, (\alpha(x) \wedge \forall y \leq x \, \neg \beta(y))$, 
\item 
$\varphi \preccurlyeq \psi : \equiv \exists x \, (\alpha(x) \wedge \forall y < x \, \neg \beta(y))$.
\end{itemize}
We may apply the witness comparison notation to formulas of the forms $\neg \forall x \varphi(x)$ and $\exists x \alpha(x) \vee \exists x \beta(x)$ by considering $\exists x \neg \varphi(x)$ and $\exists x (\alpha(x) \vee \beta(x))$, respectively. 
We will freely use the following proposition without referring to it: 

\begin{prop}[cf.~Lindstr\"om {\cite[Lemma 1.3]{Lin}}]\label{wc}
For any formulas $\varphi \equiv \exists x \alpha(x)$ and $\psi \equiv \exists x \beta(x)$, the following clauses hold: 
\begin{enumerate}
\item 
$\PA \vdash \varphi \prec \psi \to \varphi \preccurlyeq \psi$.
\item 
$\PA \vdash \varphi \vee \psi \to (\varphi \prec \psi) \vee (\psi \preccurlyeq \varphi)$.
\item 
$\PA \vdash \neg \bigl((\varphi \prec \psi) \wedge (\psi \preccurlyeq \varphi) \bigr)$.
\item 
$\PA \vdash \varphi \wedge \neg \psi \to \varphi \prec \psi$.
\end{enumerate}
\end{prop}

Let $\Prf_T(x,y)$ be a $\Delta_1(\PA)$ formula naturally expressing that ``$y$ is a $T$-proof of $x$''. 
Let $\PR_T(x)$ be the $\Sigma_1$ formula $\exists y \Prf_T(x,y)$.
For each $\Gamma \in \{\Sigma_n, \Pi_n\}$,
let $\Prf_{T}^{\Gamma}(x,y)$ denote the \textit{relativized proof predicate of $T$ with respect to $\Gamma$} defined as follows (See Lindstr\"{o}m \cite{Lin}): 
\[
\Prf_{T}^{\Gamma}(x,y) \equiv \exists u \leq y\, 
\bigl(\Gamma(u) \wedge \True_{\Gamma}(u) \wedge \Prf_T(u \dot{\to} x, y) \bigr).
\]
Here, $\Gamma(u)$ is a $\Delta_1(\PA)$ formula naturally expressing that ``$u$ is a $\Gamma$ sentence'', $\True_{\Gamma}(x)$ is a $\Gamma$ formula expressing that ``$x$ is a true $\Gamma$ sentence'' (cf.~H\'{a}jek and Pudl\'ak \cite{HP}), and $u \dot{\to} x$ is a primitive recursive term corresponding to the primitive function calculating the G\"{o}del number of $\varphi \to \psi$ from those of $\varphi$ and $\psi$.
We also introduce the relativized proof predicates with respect to $\Sigma_n \land \Pi_n$ and $\Delta_n$ as follows (cf.~H\'{a}jek \cite{Haj79} and Kogure and Kurahashi \cite{KK}): 
\begin{align*}
\Prf_{T}^{\Sigma_n \wedge \Pi_n}(x,y) : \equiv \exists u, v \leq y \, \bigl(\Sigma_n(u) & \wedge \Pi_n(v) \wedge \True_{\Sigma_n}(u) \\ & \wedge \True_{\Pi_n}(v) \wedge \Prf_T(u \dot{\wedge} v \dot{\to} x, y) \bigr).
\end{align*}
\begin{align*}
\Prf_{T}^{(\Delta_n, \Sigma_n)} (x,y) : \equiv & \exists u, v, w \leq y \, \Bigl(\Sigma_n(u) \wedge \Pi_n(v) \\
& \quad \wedge \Prf_{T}(u \dot{\leftrightarrow} v, w) \wedge \True_{\Sigma_n}(u) \wedge \Prf_{T}(u \dot{\to} x, y)\Bigr). 
\end{align*}
\begin{align*}
\Prf_{T}^{(\Delta_n, \Pi_n)} (x,y) : \equiv & \exists u, v, w \leq y \, \Bigl(\Sigma_n(u) \wedge \Pi_n(v) \\
& \quad \wedge \Prf_{T}(u \dot{\leftrightarrow} v, w) \wedge \True_{\Pi_n}(v) \wedge \Prf_{T}(v \dot{\to} x, y)\Bigr). 
\end{align*}
For $\Gamma \in \{\Sigma_n, \Pi_n\}$, the formulas $\Prf_{T}^{\Gamma}(x,y)$, $\Prf_{T}^{\Sigma_n \wedge \Pi_n}(x,y)$, $\Prf_{T}^{(\Delta_n, \Sigma_n)} (x,y)$, and $\Prf_{T}^{(\Delta_n, \Pi_n)} (x,y)$ are $\PA$-provably equivalent to $\Gamma$, $\Sigma_{n+1}$, $\Sigma_n$, and $\Pi_n$ formulas, respectively.
For any $\Lambda$, we define $\PR_{T}^{\Lambda}(x) : \equiv \exists y \, \Prf_{T}^{\Lambda}(x,y)$. 
We will also freely use the following proposition without referring to it:  

\begin{prop}[cf.~Lindstr\"{o}m \cite{Lin} and Kogure and Kurahashi \cite{KK}]\label{small_rfn}
Let $\Gamma$ denote one of $\Sigma_n$, $\Pi_n$, and $\Sigma_n \land \Pi_n$. 
Let $\Lambda$ denote one of $\Sigma_n$, $\Pi_n$, $\Sigma_n \land \Pi_n$, $(\Delta_n, \Sigma_n)$, and $(\Delta_n, \Pi_n)$. 
\begin{enumerate}
\item 
For any sentence $\varphi$ and $p \in \omega$, we have $T \vdash \Prf_{T}^{\Lambda}(\gn{\varphi}, \num{p}) \to \varphi$.
\item 
For any $\Gamma$ sentence $\gamma$, if $T + \gamma \vdash \varphi$, then $\PA + \gamma \vdash \Prf_{T}^{\Gamma}(\gn{\varphi}, \num{p})$ for some $p \in \omega$.
\item 
For any subtheory $U$ of $T$ and any $\Delta_n(U)$ sentence $\delta$, if $U + \delta \vdash \varphi$, then $\PA + U + \delta \vdash \Prf_{T}^{(\Delta_n, \Sigma_n)}(\gn{\varphi}, \num{p})$ for some $p \in \omega$ and $\PA + U + \delta \vdash \Prf_{T}^{(\Delta_n, \Pi_n)}(\gn{\varphi}, \num{q})$ for some $q \in \omega$.
\end{enumerate}
\end{prop}

From now on, we assume that $\Theta$ always denotes one of $\Sigma_n$, $\Sigma_n \land \Pi_n$, $\mathcal{B}(\Sigma_n)$, $\Delta_{n+1}(\PA)$, and $\Sigma_{n+1}$. 
Also, $\Gamma$ and $\Lambda$ always denote one of $\Delta_n$, $\SP$, $\Sigma_n$, $\Pi_n$, $\Sigma_n \land \Pi_n$, and $\mathcal{B}(\Sigma_n)$.

\begin{defn}[$\Gamma$-conservativity]\label{cons}
Let $U$ and $V$ be any theories. 
\begin{itemize}
    \item For $\Gamma \notin \{\Delta_n, \SP\}$, we say that $U$ is \textit{$\Gamma$-conservative over $V$} iff for any $\Gamma$ sentence $\gamma$, if $U \vdash \gamma$, then $V \vdash \gamma$. 

    \item We say that $U$ is \textit{$\SP$-conservative over $V$} iff for any $\sigma \in \Sigma_n$ and $\pi \in \Pi_n$, if $U \vdash \sigma \land \pi$, then $V \vdash \sigma \lor \pi$. 

    \item We say that $U$ is \textit{$\Delta_n$-conservative over $V$} iff for any sentence $\delta$ which is $\Delta_n(\Th(U) \cap \Th(V))$, if $U \vdash \delta$, then $V \vdash \delta$. 
    
    \item We say that a sentence $\varphi$ is \textit{$\Gamma$-conservative over $U$} iff $U + \varphi$ is $\Gamma$-conservative over $U$. 

    \item Let $\Cons(\Gamma, U)$ denote the set of all sentences that are $\Gamma$-conservative over $U$. 
\end{itemize}
\end{defn}

The notion of $\SP$-conservativity was introduced by the authors \cite{KK}. 

\begin{fact}[{\cite[Proposition 5.2]{KK}}]\label{Fact_DA}
Let $U$ and $V$ be any theories.
\begin{enumerate}
    \item If $U$ is $\Sigma_n$-conservative over $V$, then $U$ is $\SP$-conservative over $V$. 

    \item If $U$ is $\Pi_n$-conservative over $V$, then $U$ is $\SP$-conservative over $V$. 

    \item If $U$ is $\SP$-conservative over $V$, then $U$ is $\Delta_n$-conservative over $V$. 
\end{enumerate}
\end{fact}

Figure \ref{Fig0} visualizes the implications between the variations of $\Gamma$-conservativity. 

\begin{figure}[ht]
\centering
\begin{tikzpicture}
\node (D+) at (0,1) {$\Delta_{n+1}$-cons.};
\node (B) at (0,0) {$\mathcal{B}(\Sigma_{n})$-cons.};
\node (and) at (0,-1) {$\Sigma_{n} \land \Pi_{n}$-cons.};
\node (S) at (2,-2) {$\Pi_{n}$-cons.};
\node (P) at (-2,-2) {$\Sigma_{n}$-cons.};
\node (SP) at (0,-3) {$\SP$-cons.};
\node (D) at (0,-4) {$\Delta_{n}$-cons.};

\draw [->, double] (D+)--(B);
\draw [->, double] (B)--(and);
\draw [->, double] (and)--(S);
\draw [->, double] (and)--(P);
\draw [->, double] (S)--(SP);
\draw [->, double] (P)--(SP);
\draw [->, double] (SP)--(D);

\end{tikzpicture}
\caption{The implications between the variations of partial conservativity}\label{Fig0}
\end{figure}
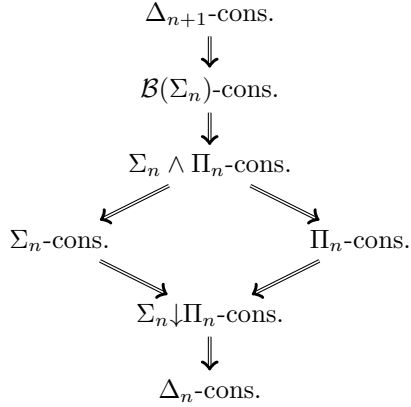

The following hereditary $\Gamma$-conservativity was introduced by Guaspari \cite{Gua}. 

\begin{defn}[Hereditary $\Gamma$-conservativity]\label{HCons}\leavevmode
\begin{itemize}
    \item We say that a sentence $\varphi$ is \textit{hereditarily $\Gamma$-conservative over $T$} iff $\varphi \in \bigcap \{\Cons(\Gamma, U) \mid T \vdash U \vdash \PA\}$. 
    
    \item Let $\HCons(\Gamma, T)$ be the set of all sentences that are hereditarily $\Gamma$-conservative over $T$. 
\end{itemize}
\end{defn}

In the sense of Definition \ref{HCons}, throughout the present paper, we assume that the term ``subtheory'' always denotes a subtheory that includes $\PA$.

\section{Doubly conservative sentences}\label{Sec:DCS}

In this section, we study $\Theta$ sentences that are doubly $(\Gamma, \Lambda)$-conservative over $U$.  

\begin{defn}[Double $(\Gamma, \Lambda)$-conservativity]
Let $U$ be any theory. 
\begin{enumerate}
    \item We say that a sentence $\varphi$ is \textit{doubly $(\Gamma, \Lambda)$-conservative} over $U$ iff $\varphi \in \Cons(\Gamma, U)$ and $\neg \varphi \in \Cons(\Lambda, U)$. 

    \item Let $\DCons(\Gamma, \Lambda; U)$ denote the set of all sentences that are doubly $(\Gamma, \Lambda)$-conservative over $U$. 
    \end{enumerate}
\end{defn}

Known results on the existence of a $\Theta$ sentence that is doubly $(\Gamma, \Lambda)$-conservative over $T$ are summarized as follows. 

\begin{fact}\label{Fact}\leavevmode
\begin{enumerate}
    \item $\Sigma_n \cap \DCons(\Pi_n, \Sigma_n; T) \neq \emptyset$. \hfill \textup{(Solovay (cf.~Lindstr\"om {\cite[Theorem 5.3]{Lin}}))}

    \item $\Delta_{n+1}(\PA) \cap \DCons(\Pi_n, \mathcal{B}(\Sigma_n); T) \neq \emptyset$. \hfill \textup{(H\'ajek \cite[Theorem 2]{Haj79})}

    \item If $T$ is $\Sigma_n$-sound, then $\Delta_{n+1}(\PA) \cap \DCons(\Sigma_n, \Sigma_n; T) = \emptyset$. \hfill \textup{(Lindstr\"om {\cite[Exercise 5.6.(b)]{Lin}})}

    \item $\mathcal{B}(\Sigma_n) \cap \DCons(\Sigma_n, \Sigma_n; T) = \mathcal{B}(\Sigma_n) \cap \DCons(\Pi_n, \Pi_n; T) = \emptyset$. \hfill \textup{(Lindstr\"om {\cite[Exercise 5.6.(c)]{Lin}})}

\end{enumerate}
\end{fact}

Notice that $\varphi \in \DCons(\Gamma, \Lambda; T)$ is equivalent to $\neg \varphi \in \DCons(\Lambda, \Gamma; T)$. 
So, $\Pi_n \cap \DCons(\Sigma_n, \Pi_n; T) \neq \emptyset$ directly follows from Fact \ref{Fact}.1.  
Then, from Fact \ref{Fact_DA}, we obtain $\Sigma_n \cap \DCons(\SP, \Sigma_n; T) \neq \emptyset$, $\Pi_n \cap \DCons(\SP, \Pi_n; T) \neq \emptyset$, and so on. 
Fact \ref{Fact}.1 turns out to be the best result for the case $\Theta \in \{\Sigma_n, \Pi_n\}$.
Also, $(\Sigma_n \land \Pi_n) \cap \DCons(\Pi_n, \Sigma_n; T) \neq \emptyset$ and $(\Sigma_n \land \Pi_n) \cap \DCons(\Sigma_n, \Pi_n; T) \neq \emptyset$ follow from Fact \ref{Fact}.1. 
In the following, we will not mention such obvious consequences of the facts and theorems.
We prove the following theorem regarding a missing case.

\begin{thm}\label{Doubly_and_SP}
$(\Sigma_n \land \Pi_n) \cap \DCons(\SP, \Sigma_n \land \Pi_n; T) \neq \emptyset$. 
\end{thm}
\begin{proof}
By Fact \ref{Fact}.1, we obtain $\varphi \in \Sigma_n \cap \DCons(\Pi_n, \Sigma_n;T)$ and $\psi \in \Pi_n \cap \DCons(\Sigma_n, \Pi_n;T +\varphi)$. 
We prove $\varphi \wedge \psi \in \DCons(\SP, \Sigma_n \land \Pi_n; T)$.

First, we show $\varphi \wedge \psi \in \Cons(\SP, T)$. 
Suppose $T + (\varphi \land \psi) \vdash \sigma \wedge \pi$ where $\sigma \in \Sigma_n$ and $\pi \in \Pi_n$. 
It follows that $T+ (\varphi \land \psi) \vdash \sigma \prec \neg \pi$.
Since $\psi \in \Cons(\Sigma_n, T+\varphi)$, we obtain $T+ \varphi \vdash \sigma \prec \neg \pi$, which implies $T+ \varphi \vdash \neg(\neg \pi \preccurlyeq \sigma)$. 
Since $\varphi \in \Cons(\Pi_n, T)$, we obtain $T \vdash \neg (\neg \pi \preccurlyeq \sigma)$, and $T + \neg \pi \vdash \sigma \prec \neg \pi$. 
It follows $T + \neg \pi \vdash \sigma$, equivalently, $T \vdash \sigma \lor \pi$. 

Second, we prove $\neg (\varphi \land \psi) \in \Cons(\Sigma_n \wedge \Pi_n;T)$. 
Let $\sigma \in \Sigma_n$ and $\pi \in \Pi_n$ be such that $T + \neg (\varphi \land \psi) \vdash \sigma \wedge \pi$. 
Then, we obtain $T + \neg \varphi \vdash \sigma$ and $T+ \neg \psi \vdash \pi$. Since $\neg \varphi \in \Cons(\Sigma_n,T)$, it follows that $T \vdash \sigma$.
Since $T + \neg \psi + \varphi \vdash \pi$, $\varphi \in \Cons(\Pi_n,T)$, and $\neg \psi \in \Cons(\Pi_n, T+ \varphi)$, we obtain $T \vdash \pi$. 
We conclude $T \vdash \sigma \wedge \pi$.
\end{proof}

Table \ref{table0} summarizes the $\Theta = \Sigma_n \land \Pi_n$ case. 

\begin{table}[ht]
\centering
\scriptsize{
\begin{tabular}{|c||c|c|c|c|}
\hline
\diagbox{$\Gamma$}{$\Lambda$} & $\SP$ & $\Pi_n$ & $\Sigma_n$ & $\Sigma_n \land \Pi_n$ \\
\hline

\hline 
$\Sigma_n$ &
\multicolumn{1}{>{\columncolor[gray]{1}}c|}{$\checkmark$}
& \multicolumn{1}{>{\columncolor[gray]{1}}c|}{\begin{tabular}{c} $\checkmark$ \\ Fact \ref{Fact}.1 \\ Solovay \end{tabular}}
& \multicolumn{1}{>{\columncolor[gray]{0.5}}c|}{\begin{tabular}{c} $\times$ \\ Fact \ref{Fact}.4 \\ Lindstr\"om \end{tabular}}
& \multicolumn{1}{>{\columncolor[gray]{0.5}}c|}{$\times$} \\

\hline 
$\Pi_n$ & 
\multicolumn{1}{>{\columncolor[gray]{1}}c|}{$\checkmark$}
& \multicolumn{1}{>{\columncolor[gray]{0.5}}c|}{\begin{tabular}{c} $\times$ \\ Fact \ref{Fact}.4 \\ Lindstr\"om \end{tabular}}
& \multicolumn{1}{>{\columncolor[gray]{1}}c|}{\begin{tabular}{c} $\checkmark$ \\ Fact \ref{Fact}.1 \\ Solovay \end{tabular}}
& \multicolumn{1}{>{\columncolor[gray]{0.5}}c|}{$\times$} \\

\hline
$\SP$ & 
\multicolumn{1}{>{\columncolor[gray]{1}}c|}{$\checkmark$}
& \multicolumn{1}{>{\columncolor[gray]{1}}c|}{$\checkmark$}
& \multicolumn{1}{>{\columncolor[gray]{1}}c|}{$\checkmark$}
& \multicolumn{1}{>{\columncolor[gray]{1}}c|}{\begin{tabular}{c} $\checkmark$ \\ Thm.~\ref{Doubly_and_SP} \end{tabular}}\\

\hline
\end{tabular}
}
\caption{The existence of a $\Sigma_n \land \Pi_n$ sentence thah is doubly $(\Gamma, \Lambda)$-conservative over $T$}\label{table0}
\end{table}

It immediately follows from Theorem \ref{Doubly_and_SP} that there exists a $\Sigma_n \lor \Pi_n$ sentence $\varphi$ such that $\varphi \in \DCons(\Sigma_n \land \Pi_n, \SP; T)$. 
Table \ref{table1} then summarizes the $\Theta = \mathcal{B}(\Sigma_n)$ case.

\begin{table}[ht]
\centering
\scriptsize{
\begin{tabular}{|c||c|c|c|c|}
\hline
\diagbox{$\Gamma$}{$\Lambda$} & $\SP$ & $\Pi_n$ & $\Sigma_n$ & $\Sigma_n \land \Pi_n$ \\
\hline

\hline 
$\Sigma_n \land \Pi_n$ & 
\multicolumn{1}{>{\columncolor[gray]{1}}c|}{\begin{tabular}{c} $\checkmark$ \\ Thm.~\ref{Doubly_and_SP} \end{tabular}}
& \multicolumn{1}{>{\columncolor[gray]{0.5}}c|}{$\times$}
& \multicolumn{1}{>{\columncolor[gray]{0.5}}c|}{$\times$}
& \multicolumn{1}{>{\columncolor[gray]{0.5}}c|}{$\times$} \\

\hline 
$\Sigma_n$ &
\multicolumn{1}{>{\columncolor[gray]{1}}c|}{$\checkmark$}
& \multicolumn{1}{>{\columncolor[gray]{1}}c|}{\begin{tabular}{c} $\checkmark$ \\ Fact \ref{Fact}.1 \\ Solovay \end{tabular}}
& \multicolumn{1}{>{\columncolor[gray]{0.5}}c|}{\begin{tabular}{c} $\times$ \\ Fact \ref{Fact}.4 \\ Lindstr\"om \end{tabular}}
& \multicolumn{1}{>{\columncolor[gray]{0.5}}c|}{$\times$} \\

\hline 
$\Pi_n$ & 
\multicolumn{1}{>{\columncolor[gray]{1}}c|}{$\checkmark$}
& \multicolumn{1}{>{\columncolor[gray]{0.5}}c|}{\begin{tabular}{c} $\times$ \\ Fact \ref{Fact}.4 \\ Lindstr\"om \end{tabular}}
& \multicolumn{1}{>{\columncolor[gray]{1}}c|}{\begin{tabular}{c} $\checkmark$ \\ Fact \ref{Fact}.1 \\ Solovay \end{tabular}}
& \multicolumn{1}{>{\columncolor[gray]{0.5}}c|}{$\times$} \\

\hline
$\SP$ & 
\multicolumn{1}{>{\columncolor[gray]{1}}c|}{$\checkmark$}
& \multicolumn{1}{>{\columncolor[gray]{1}}c|}{$\checkmark$}
& \multicolumn{1}{>{\columncolor[gray]{1}}c|}{$\checkmark$}
& \multicolumn{1}{>{\columncolor[gray]{1}}c|}{\begin{tabular}{c} $\checkmark$ \\ Thm.~\ref{Doubly_and_SP} \end{tabular}}\\

\hline
\end{tabular}
}
\caption{The existence of a $\mathcal{B}(\Sigma_n)$ sentence that is doubly $(\Gamma, \Lambda)$-conservative over $T$}\label{table1}
\end{table}

We prove there always exists a $\Sigma_{n+1}$ sentence with the strongest double conservation property for the $n$-th level of arithmetical hierarchy.

\begin{thm}\label{Doubly_B_B}
$\Sigma_{n+1} \cap \DCons(\mathcal{B}(\Sigma_n), \mathcal{B}(\Sigma_n); T) \neq \emptyset$. 
\end{thm}
\begin{proof}
By the Fixed Point Theorem (cf.~Lindstr\"om \cite{Lin}), we can effectively find a $\Sigma_{n+1}$ sentence $\varphi$ satisfying the following $\PA$-provable equivalence: 
\[
\PA \vdash \varphi \leftrightarrow \PR_T^{\Sigma_n \land \Pi_n}(\gn{\neg \varphi}) \preccurlyeq \PR_T^{\Sigma_n \land \Pi_n}(\gn{\varphi}).
\]

We prove $\varphi \in \Cons(\mathcal{B}(\Sigma_n),T)$. 
It suffices to prove that $\varphi \in \Cons(\Sigma_n \lor \Pi_n, T)$. 
Let $\sigma \in \Sigma_n$ and $\pi \in \Pi_n$ be such that $T+ \varphi \vdash \sigma \vee \pi$. 
Then, $T+ \neg \sigma + \neg \pi \vdash \neg \varphi$ and there exists a natural number $p \in \omega$ such that $\PA + \neg \sigma + \neg \pi \vdash \Prf_T^{\Sigma_n \wedge \Pi_n}(\gn{\neg \varphi}, \num{p})$. 
Since $T+ \neg \varphi \vdash \forall z < \num{p} \, \neg \Prf_T^{\Sigma_n \wedge \Pi_n}(\gn{\varphi},z)$, it follows that $T+ \neg \varphi + \neg \sigma + \neg \pi \vdash \varphi$, and thus $T+ \neg \sigma + \neg \pi \vdash \varphi$. 
By combining this with $T+ \varphi \vdash \sigma \vee \pi$, we conclude $T \vdash \sigma \vee \pi$.

$\neg \varphi \in \Cons(\mathcal{B}(\Sigma_n),T)$ is proved similarly.
\end{proof}

\subsection{$\Delta_{n+1}(U)$ sentences}

In this subsection, we consider the $\Theta = \Delta_{n+1}(U)$ case for subtheories $U$ of $T$. 
Fact \ref{Fact}.2 states that there always exists a $\Delta_{n+1}(\PA)$ sentence $\varphi$ such that $\varphi \in \DCons(\Pi_n, \mathcal{B}(\Sigma_n); T)$, and equivalently, $\neg \varphi \in \DCons(\mathcal{B}(\Sigma_n), \Pi_n; T)$.
On the other hand, it follows from Fact \ref{Fact}.3 that there exists not always a $\Delta_{n+1}(\PA)$ sentence $\varphi \in \DCons(\Sigma_n, \Sigma_n; T)$. 
Here, we refine Fact \ref{Fact}.3 by specifying the condition for which $\Delta_{n+1}(U) \cap \DCons(\Sigma_n, \Sigma_n; T) = \emptyset$ holds. 
For this purpose, we introduce the following definition. 

\begin{defn}\label{Consis}
Let $U$ be any theory. 
\begin{itemize}
    \item $T$ is said to be \textit{$\Gamma$-consistent over $U$} iff there is no $\Gamma$ formula $\varphi(x)$ such that $U \vdash \exists x \, \varphi(x)$ and $T \vdash \neg \varphi(\num{k})$ for all $k \in \omega$. 

    \item $T$ is said to be \textit{$\Gamma$-consistent} iff $T$ is $\Gamma$-consistent over $T$. 
\end{itemize}
\end{defn}

$\Sigma_n$-consistency is exactly Kreisel's $n$-consistency \cite{Kre57} that is a stratification of G\"odel's $\omega$-consistency. 
Our Definition \ref{Consis} provides a relativized version of Kreisel's $n$-consistency. 
We then prove the following refinement of Fact \ref{Fact}.3. 

\begin{thm}\label{Consis_Cons}
Let $U$ be any subtheory of $T$. 
The following are equivalent:
\begin{enumerate}
    \item $T$ is $\Sigma_{n+1}$-consistent over $U$.

    \item $T$ is $\Pi_n$-consistent over $U$.

    \item $\Delta_{n+1}(U) \cap \DCons(\Sigma_n, \Sigma_n; T) = \emptyset$. 

    \item $\Delta_{n+1}(U) \cap \DCons(\mathcal{B}(\Sigma_n), \mathcal{B}(\Sigma_n); T) = \emptyset$. 

\end{enumerate}
\end{thm}
\begin{proof}
$(1 \Leftrightarrow 2)$: This equivalence is easily proved by using a pairing function. 

$(2 \Rightarrow 3)$: 
We prove the contrapositive. 
Suppose that there exists a $\Delta_{n+1}(U)$ sentence $\varphi$ such that $\varphi \in \DCons(\Sigma_n, \Sigma_n;T)$. 
Since $\varphi \in \Delta_{n+1}(U)$, there exist a $\Sigma_n$ formula $\sigma(x)$ and a $\Pi_n$ formula $\pi(x)$ such that 
\[
U \vdash \varphi \leftrightarrow \forall x \, \sigma(x) \quad \text{and} \quad  U \vdash \varphi \leftrightarrow \exists x \, \pi(x).
\]
Since $U+ \varphi \vdash \exists x \left(\pi(x) \vee \neg \sigma(x) \right)$ and $U+ \neg \varphi \vdash \exists x \left(\pi(x) \vee \neg \sigma(x) \right)$, 
we obtain $U \vdash \exists x \left(\pi(x) \vee \neg \sigma(x) \right)$ by the law of excluded middle.
On the other hand, since $U + \varphi \vdash \forall x \, \sigma(x)$ and $U + \neg \varphi \vdash \forall x\, \neg \pi(x)$, we obtain $T + \varphi \vdash \sigma(\num{k})$ and $T+ \neg \varphi \vdash \neg \pi(\num{k})$ for all $k \in \omega$. 
Then, it follows from $\varphi \in \DCons(\Sigma_n, \Sigma_n;T)$ that $T \vdash \sigma(\num{k}) \wedge \neg \pi(\num{k})$ for all $k \in \omega$.
Since $\pi(x) \lor \neg \sigma(x)$ is a $\Pi_n$ formula, we conclude that $T$ is not $\Pi_n$-consistent over $U$. 

\medskip

$(3 \Rightarrow 4)$: Obvious.

\medskip

$(4 \Rightarrow 2)$: We prove the contrapositive.
Suppose that $T$ is not $\Pi_{n}$-consistent over $U$. 
Then, there exists a $\Pi_{n}$ formula $\pi(x)$ such that 
\[
U \vdash \exists x \, \pi(x) \quad \text{and} \quad T \vdash \neg \pi(\num{k}) \ \text{for all} \ k \in \omega.
\]

Let $\varphi$ be a $\Sigma_{n+1}$ sentence satisfying
\[
\PA \vdash \varphi \leftrightarrow \PR_{T}^{\Sigma_n \wedge \Pi_n}(\gn{\neg \varphi}) \preccurlyeq \bigl(\exists x \, \pi(x) \vee \PR_{T}^{\Sigma_n \wedge \Pi_n}(\gn{\varphi})\bigr).
\]
Since $U \vdash \exists x \, \pi(x)$, we obtain
\[
U \vdash \varphi \leftrightarrow \neg \Bigl( \bigl(\exists x \, \pi(x) \vee \PR_{T}^{\Sigma_n \wedge \Pi_n}(\gn{\varphi}) \bigr) \prec \PR_{T}^{\Sigma_n \wedge \Pi_n}(\gn{\neg \varphi}) \Bigr).
\]
Thus, $\varphi$ is a $\Delta_{n+1}(U)$ sentence.
Then, $\varphi \in \DCons(\mathcal{B}(\Sigma_n), \mathcal{B}(\Sigma_n); T)$ is proved in the same way as in the proof of Theorem \ref{Doubly_B_B}, so we omit its proof.
\end{proof}

Table \ref{table2} summarizes the case $\Theta = \Delta_{n+1}(\PA)$. 

\begin{table}[ht]
\centering
\scriptsize{
\begin{tabular}{|c||c|c|c|c|}
\hline
\diagbox{$\Gamma$}{$\Lambda$} & $\Pi_n$ & $\Sigma_n$ & $\Sigma_n \land \Pi_n$ & $\mathcal{B}(\Sigma_n)$ \\
\hline

\hline 
$\mathcal{B}(\Sigma_n)$ & 
\multicolumn{1}{>{\columncolor[gray]{1}}c|}{\begin{tabular}{c} $\checkmark$ \\ Fact \ref{Fact}.2 \\ H\'ajek \end{tabular}}
& \multicolumn{1}{>{\columncolor[gray]{0.8}}c|}{\begin{tabular}{c} $\Sigma_{n+1}$-incon. \\ over $\PA$ \end{tabular}}
& \multicolumn{1}{>{\columncolor[gray]{0.8}}c|}{\begin{tabular}{c} $\Sigma_{n+1}$-incon. \\ over $\PA$ \end{tabular}}
& \multicolumn{1}{>{\columncolor[gray]{0.8}}c|}{\begin{tabular}{c} $\Sigma_{n+1}$-incon. \\ over $\PA$ \\ Thm.~\ref{Consis_Cons} \end{tabular}} \\

\hline 
$\Sigma_n \land \Pi_n$ & \multicolumn{1}{>{\columncolor[gray]{1}}c|}{$\checkmark$}
& \multicolumn{1}{>{\columncolor[gray]{0.8}}c|}{\begin{tabular}{c} $\Sigma_{n+1}$-incon. \\ over $\PA$ \end{tabular}}
& \multicolumn{1}{>{\columncolor[gray]{0.8}}c|}{\begin{tabular}{c} $\Sigma_{n+1}$-incon. \\ over $\PA$ \end{tabular}}
& \multicolumn{1}{>{\columncolor[gray]{0.8}}c|}{\begin{tabular}{c} $\Sigma_{n+1}$-incon. \\ over $\PA$ \end{tabular}} \\

\hline 
$\Sigma_n$ & 
\multicolumn{1}{>{\columncolor[gray]{1}}c|}{$\checkmark$}
& \multicolumn{1}{>{\columncolor[gray]{0.8}}c|}{\begin{tabular}{c} $\Sigma_{n+1}$-incon. \\ over $\PA$ \\ Thm.~\ref{Consis_Cons} \end{tabular}}
& \multicolumn{1}{>{\columncolor[gray]{0.8}}c|}{\begin{tabular}{c} $\Sigma_{n+1}$-incon. \\ over $\PA$ \end{tabular}}
& \multicolumn{1}{>{\columncolor[gray]{0.8}}c|}{\begin{tabular}{c} $\Sigma_{n+1}$-incon. \\ over $\PA$ \end{tabular}} \\

\hline 
$\Pi_n$ & 
\multicolumn{1}{>{\columncolor[gray]{1}}c|}{$\checkmark$}
& \multicolumn{1}{>{\columncolor[gray]{1}}c|}{$\checkmark$}
& \multicolumn{1}{>{\columncolor[gray]{1}}c|}{$\checkmark$}
& \multicolumn{1}{>{\columncolor[gray]{1}}c|}{\begin{tabular}{c} $\checkmark$ \\ Fact \ref{Fact}.2 \\ H\'ajek \end{tabular}} \\

\hline
\end{tabular}
}
\caption{The existence of a $\Delta_{n+1}(\PA)$ sentence that is doubly $(\Gamma, \Lambda)$-conservative over $T$}\label{table2}
\end{table}

The following kinds of questions naturally arise here:  
\begin{itemize}
    \item Is there a theory $T$ such that $\Delta_{n+1}(\PA) \cap \DCons(\Sigma_n, \Sigma_n; T) =  \emptyset$ but $\Delta_{n+1}(T) \cap \DCons(\Sigma_n, \Sigma_n; T) \neq \emptyset$?

    \item Is there a theory $T$ such that $\Delta_{n+1}(T) \cap \DCons(\Sigma_n, \Sigma_n; T) = \emptyset$ but $\Delta_{n+2}(\PA) \cap \DCons(\Sigma_{n+1}, \Sigma_{n+1}; T) \neq \emptyset$?
\end{itemize}
By Theorem \ref{Consis_Cons}, these questions are transformed into the following questions respectively: 
\begin{itemize}
    \item Is there a theory $T$ such that $T$ is $\Sigma_{n+1}$-consistent over $\PA$ but is not $\Sigma_{n+1}$-consistent?

    \item Is there a theory $T$ such that $T$ is $\Sigma_{n+1}$-consistent but is not $\Sigma_{n+2}$-consistent over $\PA$?
\end{itemize}
To answer these questions, we pay attention to the following Smory\'nski's characterization of $n$-consistency. 
Let $\RFN_{\Pi_n}(T)$ be the $\Pi_n$ sentence
\[
    \forall x \bigl(\PR_T(\gn{\True_{\Pi_n}(\dot{x})}) \to \True_{\Pi_n}(x) \bigr), 
\]
which is called the uniform $\Pi_n$ reflection principle for $T$. 

\begin{fact}[Smory\'nski {\cite[Theorem 1.1]{Smo77_a}}]\label{Smo_RFN}
For $n \geq 2$, the following are equivalent: 
\begin{enumerate}
    \item $T$ is $\Sigma_n$-consistent. 
    \item $T + \RFN_{\Pi_n}(T)$ is $\Sigma_2$-sound. 
\end{enumerate}
\end{fact}

We prove the following refinement of Smory\'nski's theorem. 
Our proof of Theorem \ref{RFN} merely traces Smory\'nski's proof of Theorem \ref{Smo_RFN}. 

\begin{thm}\label{RFN}
For $n \geq 2$, the following are equivalent: 
\begin{enumerate}
    \item $T$ is $\Sigma_n$-consistent over $U$. 
    \item $T$ is $\Sigma_n$-consistent over $U + \RFN_{\Pi_n}(T)$. 
    \item $U + \RFN_{\Pi_n}(T)$ is $\Sigma_2$-sound. 
\end{enumerate}
\end{thm}
\begin{proof}
$(1 \Rightarrow 2)$: Suppose that $T$ is $\Sigma_n$-consistent over $U$. 
Let $\varphi(x)$ be any $\Sigma_n$ formula and suppose $U + \RFN_{\Pi_n}(T) \vdash \exists x \, \varphi(x)$.
Since $U \vdash \bigl( \exists w\, \Prf_T(\gn{\True_{\Pi_n}(\dot{z})}, w) \to \True_{\Pi_n}(z) \bigr) \to \exists x \, \varphi(x)$, we have
\[
   U \vdash \exists w, z, x \, \Bigl( \bigl(\Prf_T(\gn{\True_{\Pi_n}(\dot{z})},w) \to \True_{\Pi_n}(z)\bigr) \to \varphi(x) \Bigr).
\]
Then, by using a computable bijection $\langle \cdot, \cdot, \cdot \rangle : \mathbb{N}^3 \to \mathbb{N}$, we obtain
\[
    U \vdash \exists v \, \Bigl[\exists w, z, x \leq v \, \Bigl( v = \langle w, z, x \rangle \land \Bigl( \bigl(\Prf_T(\gn{\True_{\Pi_n}(\dot{z})},w) \to \True_{\Pi_n}(z)\bigr) \to \varphi(x) \Bigr) \Bigr]. 
\]
Since the formula in the square brackets $\left[ \cdots \right]$ is $\Sigma_n$ and $T$ is $\Sigma_n$-consistent over $U$, there exists a natural number $m$ such that
\[
T \nvdash \forall w, z, x \leq \num{m} \, \Bigl( \num{m} = \langle w, z, x \rangle \to \Bigl( \bigl(\Prf_T(\gn{\True_{\Pi_n}(\dot{z})},w) \to \True_{\Pi_n}(z)\bigr) \land \neg \varphi(x) \Bigr). 
\]
For $i, j, k \in \omega$ with $m = \langle i, j, k \rangle$, this is equivalent to 
\[
T \nvdash \bigl(\Prf_T(\gn{\True_{\Pi_n}(\num{j})}, \num{i}) \to \True_{\Pi_n}(\num{j})\bigr) \land \neg \varphi(\num{k}). 
\]
Since $T \vdash \Prf_T(\gn{\True_{\Pi_n}(\num{j})}, \num{i}) \to \True_{\Pi_n}(\num{j})$, we conclude $T \nvdash \neg \varphi(\num{k})$.

\medskip

$(2 \Rightarrow 3)$: 
Suppose that $T$ is $\Sigma_n$-consistent over $U + \RFN_{\Pi_n}(T)$. 
Let $\pi(x)$ be a $\Pi_1$ formula and suppose $U + \RFN_{\Pi_n}(T) \vdash \exists x \, \pi(x)$. 
Since $n \geq 2$, $T$ is $\Sigma_2$-consistent over $U+ \RFN_{\Pi_n}(T)$. 
Then, we find a natural number $m \in \omega$ such that $T \nvdash \neg \, \pi(\num{m})$. 
Since $\neg \pi(\num{m})$ is a $\Sigma_1$ sentence, by $\Sigma_1$-completeness, we obtain $\mathbb{N} \models \pi(\num{m})$. 
Thus, $\mathbb{N} \models \exists x \, \pi(x)$. 
We have proved that $U + \RFN_{\Pi_n}(T)$ is $\Sigma_2$-sound.

\medskip

$(3 \Rightarrow 1)$: 
Suppose that $U + \RFN_{\Pi_n}(T)$ is $\Sigma_2$-sound.
Let $\varphi(x)$ be any $\Sigma_n$ formula such that $U \vdash \exists x \, \varphi(x)$.
Since $\PA  + \RFN_{\Pi_n}(T) \vdash \PR_T(\gn{\neg \varphi(\dot{x})}) \to \neg \varphi(x)$, we have $\PA  + \RFN_{\Pi_n}(T) \vdash \exists x \, \varphi(x) \to \exists x\, \neg \PR_T(\gn{\neg \varphi(\dot{x})})$, and hence $U + \RFN_{\Pi_n}(T) \vdash \exists x\, \neg \PR_T(\gn{\neg \varphi(\dot{x})})$. 
By the supposition, we have $\mathbb{N} \models \exists x \, \neg \PR_T(\gn{\neg \varphi(\dot{x})})$, that is, there exists a natural number $m$ such that $T \nvdash \neg \varphi(\num{m})$.
Thus, $T$ is $\Sigma_n$-consistent over $U$.
\end{proof}

We then obtain the following equivalences and implications. 

\begin{prop}\label{Eq_Impl}
Let $U$ be any subtheory of $T$.  
\begin{enumerate}
    \item $T$ is $\Sigma_1$-consistent over $U$ if and only if $U$ is $\Sigma_1$-sound. 

    \item $T$ is $\Sigma_2$-consistent over $U$ if and only if $T$ is $\Sigma_1$-sound and $U$ is $\Sigma_2$-sound. 

    \item $T$ is $\Sigma_3$-consistent over $\PA$ if and only if $T$ is $\Sigma_2$-sound.

    \item If $T$ is $\Sigma_n$-consistent, then $T$ is $\Sigma_n$-consistent over $U$. 

    \item If $T$ is $\Sigma_n$-sound, then $T$ is $\Sigma_{n+1}$-consistent over $\PA$. 

    \item If $T$ is $\Sigma_{n+1}$-consistent over $\PA$, then $T$ is $\Sigma_n$-consistent. 

\end{enumerate}
\end{prop}
\begin{proof}
1. 
\begin{align*}
    U\ \text{is not}\ \Sigma_1\text{-sound} & \iff U \vdash \exists x \, \pi(x) \ \&\ \mathbb{N} \models \forall x \, \neg \pi(x)\ \text{for some}\ \Delta_0\ \text{formula}\ \pi(x), \\
    & \iff U \vdash \exists x \, \pi(x) \ \&\ \forall k \in \omega \, (T \vdash \neg \pi(\num{k}))\ \text{for some}\ \Delta_0\ \text{formula}\ \pi(x), \\
    & \iff T\ \text{is not}\ \Delta_0\text{-consistent over}\ U, \\
    & \iff T\ \text{is not}\ \Sigma_1\text{-consistent over}\ U.
\end{align*}

\medskip

2. By Theorem \ref{RFN}, it suffices to prove that $U + \RFN_{\Pi_2}(T)$ is $\Sigma_2$-sound if and only if $T$ is $\Sigma_1$-sound and $U$ is $\Sigma_2$-sound. 

$(\Rightarrow)$: Suppose $U + \RFN_{\Pi_2}(T)$ is $\Sigma_2$-sound, then it is also $\Pi_3$-sound. 
Then, the $\Pi_2$ sentence $\RFN_{\Pi_2}(T)$ is true, and hence $T$ is $\Sigma_1$-sound. 

$(\Leftarrow)$: Suppose that $T$ is $\Sigma_1$-sound and $U$ is $\Sigma_2$-sound. 
Let $\sigma$ be any $\Sigma_2$ sentence such that $U + \RFN_{\Pi_2}(T) \vdash \sigma$. 
By the $\Sigma_2$-soundness of $U$, we have $\mathbb{N} \models \RFN_{\Pi_2}(T) \to \sigma$. 
Since $\mathbb{N} \models \RFN_{\Pi_2}(T)$ by the $\Sigma_1$-soundness of $T$, we obtain $\mathbb{N} \models \sigma$. 

\medskip

3. This is proved in the similar manner as the second item. 

\medskip

4. Obvious.

\medskip

5. Suppose $T$ is $\Sigma_n$-sound. 
Then $\mathbb{N} \models \RFN_{\Pi_{n+1}}(T)$ holds, and hence $\PA + \RFN_{\Pi_{n+1}}(T)$ is in particular $\Sigma_2$-sound. 
By Theorem \ref{RFN}, $T$ is $\Sigma_{n+1}$-consistent over $\PA$. 

\medskip

6. The $n=1$ case follows from items 1 and 2. 
We may assume $n \geq 2$. 
Suppose that $T$ is $\Sigma_{n+1}$-consistent over $\PA$. 
By Theorem \ref{RFN}, $\PA + \RFN_{\Pi_n+1}(T)$ is $\Sigma_2$-sound. 
Since $\PA + \RFN_{\Pi_{n+1}}(T) \vdash \RFN_{\Pi_{n+1}}(T + \RFN_{\Pi_n}(T))$ (cf.~Smory\'nski \cite[Corollary 4.1.12]{Smo77}), we have that $\PA + \RFN_{\Pi_{n+1}}(T)$ is $\Pi_{n+1}$-conservative over $T + \RFN_{\Pi_n}(T)$. 
    Since $n \geq 2$, we see that $T + \RFN_{\Pi_n}(T)$ is $\Sigma_2$-sound. 
    By Theorem \ref{RFN} again, $T$ is $\Sigma_n$-consistent. 
\end{proof}

Figure \ref{Fig1} summarizes the equivalences and implications between the conditions we are dealing with. 
In the figure, we assume that $n \geq 3$ and $U$ is a subtheory of $T$. 

\begin{figure}[ht]
\centering
\begin{tikzpicture}

\node (n+1_T) at (-5.5,8) {$\Delta_{n+2}(T) \cap \DCons(\Sigma_{n+1}, \Sigma_{n+1}; T) = \emptyset$};
\node (n+1C_T) at (0,8) {$T$: $\Sigma_{n+1}$-con.};
\node (TT3S) at (5,8) {$T + \RFN_{\Pi_{n+1}}(T)$: $\Sigma_2$-sound};

\draw [<->, double] (n+1_T)--(n+1C_T);
\draw [<->, double] (n+1C_T)--(TT3S);

\node (n+1_U) at (-5.5,7) {$\Delta_{n+2}(U) \cap \DCons(\Sigma_{n+1}, \Sigma_{n+1}; T) = \emptyset$};
\node (n+1C_U) at (0,7) {$T$: $\Sigma_{n+1}$-con.~over $U$};
\node (UT2S) at (5,7) {$U + \RFN_{\Pi_{n+1}}(T)$: $\Sigma_2$-sound};

\draw [<->, double] (n+1_U)--(n+1C_U);
\draw [<->, double] (n+1C_U)--(UT2S);

\node (n_T) at (-5.5,6) {$\Delta_{n+1}(T) \cap \DCons(\Sigma_n, \Sigma_n; T) = \emptyset$};
\node (nC_T) at (0,6) {$T$: $\Sigma_n$-con.};
\node (TT2S) at (5,6) {$T + \RFN_{\Pi_n}(T)$: $\Sigma_2$-sound};

\draw [<->, double] (n_T)--(nC_T);
\draw [<->, double] (nC_T)--(TT2S);

\node (3_U) at (-5.5,5) {$\Delta_4(U) \cap \DCons(\Sigma_3, \Sigma_3; T) = \emptyset$};
\node (3C_U) at (0,5) {$T$: $\Sigma_3$-con.~over $U$};
\node (UT) at (5,5) {$U + \RFN_{\Pi_3}(T)$: $\Sigma_2$-sound};

\node (2_T) at (-5.5,4) {$\Delta_3(T) \cap \DCons(\Sigma_2, \Sigma_2; T) = \emptyset$};
\node (2C_T) at (0,4) {$T$: $\Sigma_2$-con.};
\node (2S) at (5,4) {$T$: $\Sigma_2$-sound};

\draw [<->, double] (3_U)--(3C_U);
\draw [<->, double] (3C_U)--(UT);
\draw [<->, double] (2_T)--(2C_T);
\draw [<->, double] (2C_T)--(2S);

\node (2_U) at (-5.5,3) {$\Delta_3(U) \cap \DCons(\Sigma_2, \Sigma_2; T) = \emptyset$};
\node (2C_U) at (0,3) {$T$: $\Sigma_2$-con.~over $U$};
\node (1S2S) at (5,3) {$T$: $\Sigma_1$-sound\ \&\ $U$: $\Sigma_2$-sound};

\draw [<->, double] (2_U)--(2C_U);
\draw [<->, double] (2C_U)--(1S2S);

\node (1_T) at (-5.5,2) {$\Delta_2(T) \cap \DCons(\Sigma_1, \Sigma_1; T) = \emptyset$};
\node (1C_T) at (0,2) {$T$: $\Sigma_1$-con.};
\node (1S) at (5,2) {$T$: $\Sigma_1$-sound};

\draw [<->, double] (1_T)--(1C_T);
\draw [<->, double] (1C_T)--(1S);

\node (1_U) at (-5.5,1) {$\Delta_2(U) \cap \DCons(\Sigma_1, \Sigma_1; T) = \emptyset$};
\node (1C_U) at (0,1) {$T$: $\Sigma_1$-con.~over $U$};
\node (0S1S) at (5,1) {$U$: $\Sigma_1$-sound};

\draw [<->, double] (1_U)--(1C_U);
\draw [<->, double] (1C_U)--(0S1S);

\draw [->, double] (5,8.7)--(TT3S);
\draw [->, double] (TT3S)--(UT2S);
\draw [->, double] (UT2S)--(TT2S);
\draw [->, double] (TT2S)--(UT);
\draw [->, double] (UT)--(2S);
\draw [->, double] (2S)--(1S2S);
\draw [->, double] (1S2S)--(1S);
\draw [->, double] (1S)--(0S1S);
\end{tikzpicture}
\caption{Equivalences and implications between the conditions}\label{Fig1}
\end{figure}
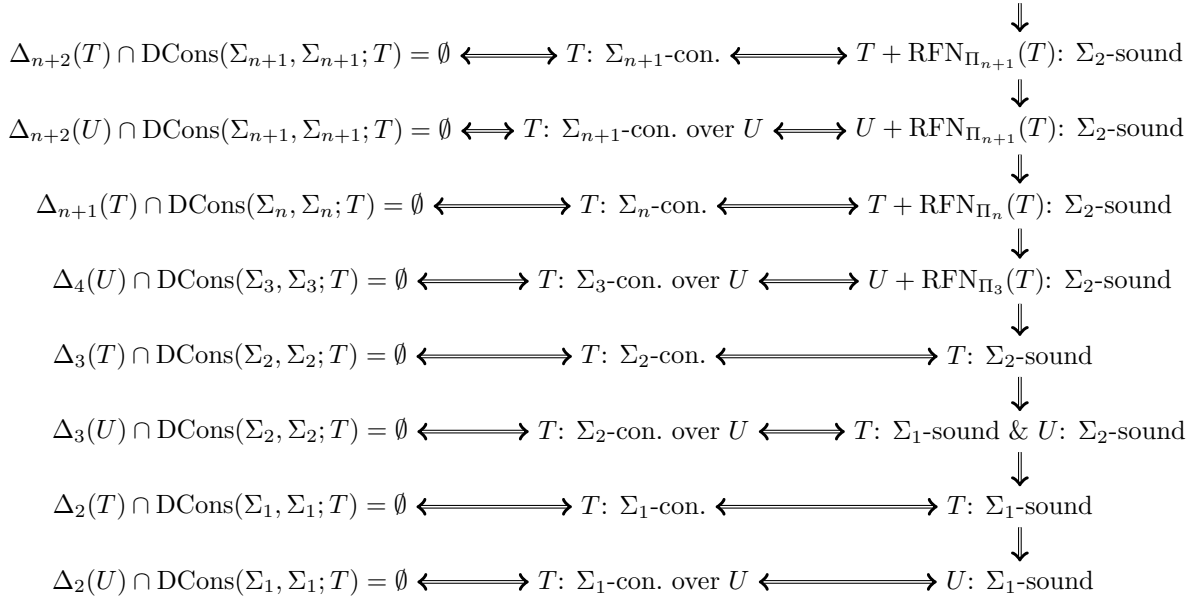

The following proposition gives an affirmative answer to our former questions.

\begin{prop}\label{CE_0}\leavevmode
\begin{enumerate}
    \item There exists a theory $T$ such that $T$ is $\Sigma_{n+1}$-consistent over $\PA$ but is not $\Sigma_{n+1}$-consistent. 

    \item For $n \geq 3$, there exists a theory $T$ such that $T$ is $\Sigma_n$-consistent but is not $\Sigma_{n+1}$-consistent over $\PA$. 

\end{enumerate}
\end{prop}
\begin{proof}
1. Let $T : = \PA + \neg \RFN_{\Pi_{n+1}}(\PA)$. 
Since $\PA \vdash \forall x \, (\Pi_{n+1}(x) \land \PR_T(x) \to \PR_{\PA}(x))$ (cf.~\cite{KOSV}), it is proved $\PA \vdash \RFN_{\Pi_{n+1}}(T) \leftrightarrow \RFN_{\Pi_{n+1}}(\PA)$. 
Then, the theory $\PA + \RFN_{\Pi_{n+1}}(T)$ is sound because it is equivalent to $\PA + \RFN_{\Pi_{n+1}}(\PA)$. 
By Theorem \ref{RFN}, $T$ is $\Sigma_{n+1}$-consistent over $\PA$. 
On the other hand, the theory $T + \RFN_{\Pi_{n+1}}(T)$ is inconsistent because it proves both $\neg \RFN_{\Pi_{n+1}}(\PA)$ and $\RFN_{\Pi_{n+1}}(\PA)$. 
By Theorem \ref{RFN} again, $T$ is not $\Sigma_{n+1}$-consistent. 

\medskip

2. Suppose that $n \geq 3$. 
Let $\varphi$ be a $\Sigma_3$ sentence satisfying the following equivalence: 
\[
    \PA \vdash \varphi \leftrightarrow \neg \RFN_{\Pi_3}(\PA + \varphi + \RFN_{\Pi_n}(\PA + \varphi)).
\]
Let $T : = \PA + \varphi$.

Suppose, towards a contradiction, that $\mathbb{N} \models \varphi$. 
Then $\mathbb{N} \models \varphi \land \RFN_{\Pi_n}(\PA + \varphi)$, and hence $\mathbb{N} \models \RFN_{\Pi_3}(\PA + \varphi + \RFN_{\Pi_n}(\PA + \varphi))$. 
This means $\mathbb{N} \models \neg \varphi$, a contradiction. 

Therefore, we obtain $\mathbb{N} \not \models \varphi$, that is, $\mathbb{N} \models \RFN_{\Pi_3}(\PA + \varphi + \RFN_{\Pi_n}(\PA + \varphi))$. 
It follows that $T + \RFN_{\Pi_n}(T)$ is $\Sigma_2$-sound. 
By Theorem \ref{RFN}, we have that $T$ is $\Sigma_n$-consistent. 

Since $n \geq 3$, we have $\varphi \in \Pi_{n+1}$. 
It follows from $T \vdash \varphi$ that $\PA \vdash \PR_T(\gn{\varphi})$, and thus
\begin{align*}
    \PA + \RFN_{\Pi_{n+1}}(T) & \vdash \varphi,\\
    \PA + \RFN_{\Pi_{n+1}}(T) & \vdash \neg \RFN_{\Pi_3}(\PA + \varphi + \RFN_{\Pi_n}(\PA + \varphi)),\\
    \PA + \RFN_{\Pi_{n+1}}(T) & \vdash \neg \RFN_{\Pi_{n+1}}(T + \RFN_{\Pi_n}(T)),\\
    \PA + \RFN_{\Pi_{n+1}}(T) & \vdash \neg \RFN_{\Pi_{n+1}}(T). 
\end{align*}
We conclude that $\PA + \RFN_{\Pi_{n+1}}(T)$ is inconsistent. 
In particular, $\PA + \RFN_{\Pi_{n+1}}(T)$ is not $\Sigma_2$-sound. 
By Theorem \ref{RFN}, we see that $T$ is not $\Sigma_{n+1}$-consistent over $\PA$. 
\end{proof}

\section{Hereditarily doubly conservative sentences}\label{Sec:HDCS}

We start this section with the following definition. 

\begin{defn}[Hereditary double $(\Gamma, \Lambda)$-conservativity]\leavevmode
\begin{enumerate}
    \item We say that a sentence $\varphi$ is \textit{hereditarily doubly $(\Gamma, \Lambda)$-conservative} over $T$ iff $\varphi \in \bigcap \{\DCons(\Gamma, \Lambda; U) \mid T \vdash U \vdash \PA\}$. 
    
    \item Let $\HDCons(\Gamma, \Lambda; T)$ be the set of all sentences that are hereditarily doubly $(\Gamma, \Lambda)$-conservative over $T$. 
\end{enumerate}
\end{defn}

The authors have shown in a previous study that $\SP$-conservativity plays an important role in the analysis of hereditary double conservativity.

\begin{fact}[Kogure and Kurahashi {\cite[Theorem 5.7]{KK}}]\label{DA0}
The following are equivalent: 
\begin{enumerate}
    \item $T$ is $\SP$-conservative over $\PA$. 

    \item $\Sigma_n \cap \HDCons(\Pi_n, \Sigma_n; T) \neq \emptyset$. 

    \item $\Pi_n \cap \HDCons(\Sigma_n, \Pi_n; T) \neq \emptyset$. 

    \item $\Th_{\Sigma_n}(T) \subseteq \HCons(\Pi_n, T)$. 
        
    \item $\Th_{\Pi_n}(T) \subseteq \HCons(\Sigma_n, T)$. 

\end{enumerate}
\end{fact}

Following Fact \ref{DA0}, in this section, we investigate the existence of $\Theta$ sentences that are hereditarily doubly $(\Gamma, \Lambda)$-conservative over $T$ for several triples $(\Theta, \Gamma, \Lambda)$. 
Each of the following subsections is distinguished according to $\Theta$, and in each subsection, we will provide a table summarizing the situation with respect to $\Theta$.
However, for some triples $(\Theta, \Gamma, \Lambda)$, the condition $\Theta \cap \HDCons(\Gamma, \Lambda; T) \neq \emptyset$ may be situated in a halfway position between the prominent conditions, and to describe such situations we introduce the following definition. 

\begin{defn}\label{Def_truple}
Let $\Gamma \neq \Delta_n$. 
We say that a theory $T$ is a \textit{$\Theta$-$(\Gamma, \Lambda)$ theory} iff there exists a sentence $\varphi$ satisfying the following three conditions: 
\begin{enumerate}
    \item $\varphi \in \Theta$. 

    \item $T + \varphi$ is $\Gamma$-conservative over $\PA$. 

    \item $\neg \varphi \in \HCons(\Lambda, T)$.
\end{enumerate}
\end{defn}

We then immediately get the following proposition. 

\begin{prop}\label{Prop_triple}
If $T$ is a $\Theta$-$(\Gamma, \Lambda)$ theory, then $\Theta \cap \HDCons(\Gamma, \Lambda; T) \neq \emptyset$. 
\end{prop}

It will be proved that the converse implication of Proposition \ref{Prop_triple} is not case in general (Corollary \ref{Cor_triple}). 
The following Propositions \ref{DA_P}, \ref{Pi_S}, and \ref{Pi_P} provide some examples of triples $(\Theta, \Gamma, \Lambda)$ such that the condition `$T$ is a $\Theta$-$(\Gamma, \Lambda)$ theory' is not situated in a halfway position.

\begin{prop}\label{DA_P}
The following are equivalent: 
\begin{enumerate}
    \item $T$ is $\SP$-conservative over $\PA$. 

    \item $T$ is a $\Pi_n$-$(\SP, \Pi_n)$ theory. 
\end{enumerate}
\end{prop}
\begin{proof}
$(1 \Rightarrow 2)$: 
Suppose that $T$ is $\SP$-conservative over $\PA$. 
By Fact \ref{DA0}, we find a $\Pi_n$ sentence $\varphi$ such that $\varphi \in \HDCons(\Sigma_n, \Pi_n; T)$. 
It suffices to prove that $T + \varphi$ is $\SP$-conservative over $\PA$. 
Let $\sigma \in \Sigma_n$ and $\pi \in \Pi_n$ be such that $T + \varphi \vdash \sigma \land \pi$. 
Then $T + \varphi \vdash \sigma \prec \neg \pi$. 
Since $\varphi \in \Cons(\Sigma_n, T)$, we have $T \vdash \sigma \prec \neg \pi$. 
Since $T$ also proves $\neg (\neg \pi \preccurlyeq \sigma)$ and $T$ is $\SP$-conservative over $\PA$, we get $\PA \vdash (\sigma \prec \neg \pi) \lor \neg (\neg \pi \preccurlyeq \sigma)$. 
Thus, $\PA \vdash \neg (\neg \pi \preccurlyeq \sigma)$. 
It follows that $\PA \vdash \sigma \lor \pi$. 

\medskip

$(2 \Rightarrow 1)$: Trivial. 
\end{proof}

\begin{prop}\label{Pi_S}
The following are equivalent: 
\begin{enumerate}
    \item $T$ is $\Sigma_n$-conservative over $\PA$. 

    \item $T$ is a $\Pi_n$-$(\Sigma_n, \Pi_n)$ theory. 
\end{enumerate}
\end{prop}
\begin{proof}
$(1 \Rightarrow 2)$: 
Suppose that $T$ is $\Sigma_n$-conservative over $\PA$. 
Since $T$ is $\SP$-conservative over $\PA$, we find a $\Pi_n$ sentence $\varphi$ such that $\varphi \in \HDCons(\Sigma_n, \Pi_n; T)$ by Fact \ref{DA0}. 
It is easy to see that $T + \varphi$ is $\Sigma_n$-conservative over $\PA$. 

\medskip

$(2 \Rightarrow 1)$: Trivial. 
\end{proof}

The following proposition is also proved in the same way as in the proof of Proposition \ref{Pi_S}. 

\begin{prop}\label{Pi_P}
The following are equivalent: 
\begin{enumerate}
    \item $T$ is $\Pi_n$-conservative over $\PA$. 

    \item $T$ is a $\Sigma_n$-$(\Pi_n, \Sigma_n)$ theory. 
\end{enumerate}
\end{prop}

Before proceeding with the analysis for each individual $\Theta$, we prove the following useful proposition.

\begin{prop}\label{G_G}
Let $\Gamma \in \{\SP, \Sigma_n, \Pi_n, \Sigma_n \land \Pi_n, \mathcal{B}(\Sigma_n)\}$.
If $\HDCons(\Gamma, \Gamma; T) \neq \emptyset$, then $T$ is $\Gamma$-conservative over $\PA$.
\end{prop}
\begin{proof}
We give only a proof of the case $\Gamma \neq \SP$. 
The case $\Gamma = \SP$ is proved in the similar way. 
Let $\varphi$ be such that $\varphi \in \HCons(\Gamma, T)$ and $\neg \varphi \in \HCons(\Gamma, T)$. 
Suppose that $\psi \in \Gamma$ and $T \vdash \psi$. 
Since $\PA + (\psi \lor \neg \varphi)$ is a subtheory of $T$ and $\PA + (\psi \lor \neg \varphi) + \varphi \vdash \psi$, we obtain $\PA + (\psi \lor \neg \varphi) \vdash \psi$ because $\varphi \in \HDCons(\Gamma, T)$. 
Hence, $\PA + \neg \varphi \vdash \psi$. 
It also follows from $\neg \varphi \in \HCons(\Gamma, T)$ that $\PA \vdash \psi$. 
\end{proof}

\subsection{$\Sigma_n$ sentences}

By combining Fact \ref{DA0} and Proposition \ref{G_G}, we obtain the following corollary.

\begin{cor}\label{DA1}
The following are equivalent: 
\begin{enumerate}
    \item $T$ is $\SP$-conservative over $\PA$. 

    \item $\HDCons(\SP, \SP; T) \neq \emptyset$. 
\end{enumerate}
\end{cor}

We refine Corollary \ref{DA1} by taking into account $\Delta_n$ as follows. 

\begin{thm}\label{HDelta}
The following are equivalent: 
\begin{enumerate}
    
    \item $T$ is $\SP$-conservative over $\PA$. 
    
    \item $\HDCons(\Delta_n, \Delta_n; T) \neq \emptyset$. 

    \item For any subtheory $U$ of $T$, we have that $T$ is $\Delta_n(U)$-conservative over $U$.

\end{enumerate}
\end{thm}
\begin{proof}
$(1 \Rightarrow 2)$: This is immediate from Facts \ref{Fact_DA} and \ref{DA0}. 

\medskip

$(2 \Rightarrow 3)$: 
Suppose $\varphi \in \HDCons(\Delta_n, \Delta_n; T)$. 
Let $U$ be a subtheory of $T$ and $\delta \in \Delta_n(U)$ be such that $T \vdash \delta$.
Since $U + (\varphi \vee \delta)$ is a subtheory of $T$
and $U + (\varphi \vee \delta) + \neg \varphi \vdash \delta$, we obtain $U+ (\varphi \vee \delta) \vdash \delta$ because $\neg \varphi \in \HCons(\Delta_n,T)$. 
In particular, $U+ \varphi \vdash \delta$. 
Since $\varphi \in \HCons(\Delta_n,T)$, it follows that $U \vdash \delta$.

\medskip

$(3 \Rightarrow 1)$:
Suppose the property stated in the third clause holds. 
Let $\sigma \in \Sigma_n$ and $\pi \in \Pi_n$ be such that $T \vdash \sigma \land \pi$. 
We have $T \vdash \sigma \prec \neg \pi$. 
Since $\PA + (\sigma \lor \neg \pi) \vdash \sigma \prec \neg \pi \leftrightarrow \neg (\neg \pi \preccurlyeq \sigma)$, we have that $\sigma \prec \neg \pi$ is $\Delta_n$ over a subtheory $\PA + (\sigma \lor \neg \pi)$ of $T$. 
By the supposition, $\PA + (\sigma \lor \neg \pi) \vdash \sigma \prec \neg \pi$. 
In particular, $\PA + \neg \pi \vdash \sigma$, and equivalently, $\PA \vdash \sigma \lor \pi$. 
\end{proof}

\begin{table}[ht]
\centering
\scriptsize{
\begin{tabular}{|c||c|c|c|}
\hline
\diagbox{$\Gamma$}{$\Lambda$} & $\Delta_n$ & $\SP$ & $\Sigma_n$ \\
\hline

\hline 
$\Pi_n$
& \multicolumn{1}{>{\columncolor[gray]{0.95}}c|}{$\SP$}
& \multicolumn{1}{>{\columncolor[gray]{0.95}}c|}{$\SP$}
& \multicolumn{1}{>{\columncolor[gray]{0.95}}c|}{\begin{tabular}{c} $\SP$ \\ Fact \ref{DA0} \end{tabular}} \\

\hline 
$\SP$
& \multicolumn{1}{>{\columncolor[gray]{0.95}}c|}{\begin{tabular}{c} $\SP$ \\ Cor.~\ref{DA1} \end{tabular}}
& \multicolumn{1}{>{\columncolor[gray]{0.95}}c|}{$\SP$}
& \multicolumn{1}{>{\columncolor[gray]{0.95}}c|}{$\SP$} \\

\hline 
$\Delta_n$
& \multicolumn{1}{>{\columncolor[gray]{0.95}}c|}{\begin{tabular}{c} $\SP$ \\ Thm.~\ref{HDelta} \end{tabular}}
& \multicolumn{1}{>{\columncolor[gray]{0.95}}c|}{$\SP$}
& \multicolumn{1}{>{\columncolor[gray]{0.95}}c|}{$\SP$} \\

\hline
\end{tabular}
}
\caption{The existence of a $\Sigma_n$ sentence that is hereditarily doubly $(\Gamma, \Lambda)$-conservative over $T$}\label{table3}
\end{table}

\subsection{$\Sigma_n \land \Pi_n$ sentences}

We show that $\SP$-conservativity also plays an important role in the $\Theta = \Sigma_n \land \Pi_n$ case.

\begin{thm}\label{DA2}
The following are equivalent: 
\begin{enumerate}
    \item $T$ is $\SP$-conservative over $\PA$. 

    \item $(\Sigma_n \land \Pi_n) \cap \HDCons(\Delta_n, \Sigma_n \land \Pi_n; T) \neq \emptyset$. 
\end{enumerate}
\end{thm}
\begin{proof}
$(1 \Rightarrow 2)$: Suppose that $T$ is $\SP$-conservative over $\PA$. 
By Proposition \ref{DA_P}, we find a $\Pi_n$ sentence $\psi$ such that $T + \psi$ is $\SP$-conservative over $\PA$ and $\neg \psi \in \HCons(\Pi_n, T)$. 
Also, by Fact \ref{DA0}, we get a $\Sigma_n$ sentence $\varphi$ such that $\varphi \in \HDCons(\Pi_n, \Sigma_n; T + \psi)$. 
We prove that the $\Sigma_n \land \Pi_n$ sentence $\varphi \land \psi$ is in $\HDCons(\Delta_n, \Sigma_n \land \Pi_n; T)$. 

First, we prove $\varphi \land \psi \in \HCons(\Delta_n, T)$. 
Let $U$ be any subtheory of $T$ and $\delta$ be a $\Delta_n(U)$ sentence such that $U + (\varphi \land \psi) \vdash \delta$. 
Since $\delta$ is equivalent to some $\Pi_n$ sentence over $U$, we have $U + \psi \vdash \delta$. 
In addition, $\delta$ is equivalent to a $\Sigma_n$ sentence, and so $U \vdash \delta$. 

Second, we prove $\neg (\varphi \land \psi) \in \HCons(\Sigma_n \land \Pi_n, T)$. 
Let $U$ be any subtheory of $T$. 
Let $\sigma \in \Sigma_n$ and $\pi \in \Pi_n$ be such that $U + \neg (\varphi \land \psi) \vdash \sigma \land \pi$. 
Since $U + \neg \varphi \vdash \sigma$ and $\neg \varphi \in \HCons(\Sigma_n, T)$, we have $U \vdash \sigma$. 
Also, $U \vdash \pi$ because $U + \neg \psi \vdash \pi$ and $\neg \psi \in \HCons(\Pi_n, T)$. 
Therefore, $U \vdash \sigma \land \pi$. 

\medskip

$(2 \Rightarrow 1)$: By Theorem \ref{HDelta}. 
\end{proof}

The only remaining case is $(\Sigma_n \land \Pi_n) \cap \HDCons(\SP, \Sigma_n \land \Pi_n; T)$. 
We found the following characterization for this case.  
Here, recalling Definition \ref{Def_truple}, we say that $T$ is a $\Theta$-$(\Gamma, \Lambda)$ theory iff there exists a $\Theta$ sentence $\varphi$ such that $T + \varphi$ is $\Gamma$-conservative over $\PA$ and $\neg \varphi \in \HCons(\Lambda, T)$. 

\begin{thm}\label{H*_Sigma_DA_S}
The following are equivalent: 
\begin{enumerate}
    \item $T$ is a $\Sigma_n$-$(\SP, \Sigma_n)$ theory. 
    
    \item $(\Sigma_{n} \land \Pi_n) \cap \HDCons(\SP, \Sigma_n \land \Pi_n; T) \neq \emptyset$. 

\end{enumerate}
\end{thm}
\begin{proof}
$(1 \Rightarrow 2)$: 
Let $\varphi$ be a $\Sigma_n$ sentence such that $T + \varphi$ is $\SP$-conservative over $\PA$ and $\neg \varphi \in \HCons(\Sigma_n, T)$. 
By Fact \ref{DA0}, there exists a $\Pi_n$ sentence $\psi$ such that $\psi \in \HDCons(\Sigma_n, \Pi_n; T + \varphi)$. 
We prove that the $\Sigma_n \land \Pi_n$ sentence $\varphi \land \psi$ is in $\HDCons(\SP, \Sigma_n \land \Pi_n; T)$. 
Since $\neg (\varphi \land \psi) \in \HCons(\Sigma_n \land \Pi_n, T)$ is proved similarly as in the proof of $(1 \Rightarrow 2)$ of Theorem \ref{DA2}, it suffices to prove $\varphi \land \psi \in \HCons(\SP T)$. 

Let $U$ be any subtheory of $T$ and let $\sigma \in \Sigma_n$ and $\pi \in \Pi_n$ be such that $U + (\varphi \land \psi) \vdash \sigma \land \pi$. 
Since $U + (\varphi \land \psi) \vdash \sigma \prec \neg \pi$ and $\psi \in \HCons(\Sigma_n, T + \varphi)$, we have $U + \varphi \vdash \sigma \prec \neg \pi$. 
Then, 
\begin{equation}\label{eq0}
    U + \varphi \vdash \neg (\neg \pi \preccurlyeq \sigma). 
\end{equation}
It follows $T + \varphi \vdash \varphi \land \neg (\neg \pi \preccurlyeq \sigma)$. 
Since $T + \varphi$ is $\SP$-conservative over $\PA$, we obtain $\PA \vdash \varphi \lor \neg (\neg \pi \preccurlyeq \sigma)$. 
By combining this with (\ref{eq0}), $U \vdash \neg (\neg \pi \preccurlyeq \sigma)$. 
Then, we conclude $U \vdash \sigma \lor \pi$. 

\medskip

$(2 \Rightarrow 1)$: 
Let $\varphi \in \Sigma_n$ and $\psi \in \Pi_n$ be such that $\varphi \land \psi \in \HDCons(\SP, \Sigma_n \land \Pi_n; T)$. 

At first, we prove $\varphi \in \Cons(\Pi_n, T)$. 
Let $\pi \in \Pi_n$ be such that $T + \varphi \vdash \pi$. 
Since $T + (\varphi \land \psi) \vdash \varphi \land \pi$ and $\varphi \land \psi \in \HCons(\SP, T)$, we have $T \vdash \varphi \lor \pi$. 
By combining this with $T + \varphi \vdash \pi$, we obtain $T \vdash \pi$. 

Secondly, we prove $\neg \varphi \in \HCons(\Sigma_n, T)$. 
Let $U$ be any subtheory of $T$ and $\sigma \in \Sigma_n$ be such that $U + \neg \varphi \vdash \sigma$. 
Since $U + \neg (\varphi \land \psi) \vdash \sigma \lor \neg \psi$ and $\neg (\varphi \land \psi) \in \HCons(\Sigma_n, T)$, we get $U \vdash \sigma \lor \neg \psi$. 
Thus, $U + (\varphi \land \psi) \vdash \sigma \land \psi$. 
Since $\varphi \land \psi \in \HCons(\SP, T)$, we have $U \vdash \sigma \lor \psi$. 
By combining this with $U \vdash \sigma \lor \neg \psi$, we conclude $U \vdash \sigma$. 

Finally, we prove that $T + \varphi$ is $\SP$-conservative over $\PA$. 
Let $\sigma \in \Sigma_n$ and $\pi \in \Pi_n$ be such that $T + \varphi \vdash \sigma \land \pi$. 
Since we have already proved that $\varphi$ is $\Pi_n$-conservative over $T$, 
\[
    T \vdash \varphi \to \sigma \quad \text{and} \quad T \vdash \pi. 
\]
We have $\PA + (\pi \lor (\varphi \land \psi)) + \neg (\varphi \land \psi) \vdash \pi$. 
Since $\PA + (\pi \lor (\varphi \land \psi))$ is a subtheory of $T$ and $\neg (\varphi \land \psi) \in \HCons(\Sigma_n \land \Pi_n, T)$, we obtain $\PA + (\pi \lor (\varphi \land \psi)) \vdash \pi$. 
Hence,  
\begin{equation}\label{eq1}
    \PA + (\varphi \land \psi) \vdash \pi.     
\end{equation}

Therefore, we get $\PA + \bigl((\varphi \to \sigma) \lor \neg (\varphi \land \psi) \bigr) + (\varphi \land \psi) \vdash \sigma \land \pi$. 
Since $\PA + \bigl((\varphi \to \sigma) \lor \neg (\varphi \land \psi) \bigr)$ is a subtheory of $T$ and $\varphi \land \psi \in \HCons(\SP, T)$, we obtain $\PA + \bigl((\varphi \to \sigma) \lor \neg (\varphi \land \psi) \bigr) \vdash \sigma \lor \pi$. 
Hence, $\PA + \neg (\varphi \land \psi) \vdash \sigma \lor \pi$. 
By combining this with (\ref{eq1}), we conclude $\PA \vdash \sigma \lor \pi$. 
\end{proof}

Table \ref{table4} summarizes the $\Theta = \Sigma_n \land \Pi_n$ case. 
In the table, `$\Sigma$-$(\downarrow, \Sigma)$' indicates that $T$ is a $\Sigma_n$-$(\SP, \Sigma_n)$ theory. 

\begin{table}[ht]
\centering
\scriptsize{
\begin{tabular}{|c||c|c|c|c|c|}
\hline
\diagbox{$\Gamma$}{$\Lambda$} & $\Delta_n$ & $\SP$ & $\Sigma_n$ & $\Pi_n$ & $\Sigma_n \land \Pi_n$ \\
\hline

\hline 
$\Pi_n$ 
& \multicolumn{1}{>{\columncolor[gray]{0.95}}c|}{$\SP$}
& \multicolumn{1}{>{\columncolor[gray]{0.95}}c|}{$\SP$}
& \multicolumn{1}{>{\columncolor[gray]{0.95}}c|}{\begin{tabular}{c} $\SP$ \\ Fact \ref{DA0} \end{tabular}}
& \multicolumn{1}{>{\columncolor[gray]{0.5}}c|}{\begin{tabular}{c} $\times$ \\ Fact \ref{Fact}.4 \\ Lindstr\"om \end{tabular}}
& \multicolumn{1}{>{\columncolor[gray]{0.5}}c|}{$\times$} \\

\hline 
$\Sigma_n$ 
& \multicolumn{1}{>{\columncolor[gray]{0.95}}c|}{$\SP$}
& \multicolumn{1}{>{\columncolor[gray]{0.95}}c|}{$\SP$} 
& \multicolumn{1}{>{\columncolor[gray]{0.5}}c|}{\begin{tabular}{c} $\times$ \\ Fact \ref{Fact}.4 \\ Lindstr\"om \end{tabular}}
& \multicolumn{1}{>{\columncolor[gray]{0.95}}c|}{\begin{tabular}{c} $\SP$ \\ Fact \ref{DA0} \end{tabular}}
& \multicolumn{1}{>{\columncolor[gray]{0.5}}c|}{$\times$} \\

\hline 
$\SP$ & 
\multicolumn{1}{>{\columncolor[gray]{0.95}}c|}{$\SP$} 
& \multicolumn{1}{>{\columncolor[gray]{0.95}}c|}{$\SP$}
& \multicolumn{1}{>{\columncolor[gray]{0.95}}c|}{$\SP$} & \multicolumn{1}{>{\columncolor[gray]{0.95}}c|}{$\SP$}
& \multicolumn{1}{>{\columncolor[gray]{0.76}}c|}{\begin{tabular}{c} $\Sigma$-$(\downarrow, \Sigma)$ \\ Thm.~\ref{H*_Sigma_DA_S} \end{tabular}} \\

\hline 
$\Delta_n$ & 
\multicolumn{1}{>{\columncolor[gray]{0.95}}c|}{\begin{tabular}{c} $\SP$ \\ Thm.~\ref{HDelta} \end{tabular}}
& \multicolumn{1}{>{\columncolor[gray]{0.95}}c|}{$\SP$}
& \multicolumn{1}{>{\columncolor[gray]{0.95}}c|}{$\SP$}
& \multicolumn{1}{>{\columncolor[gray]{0.95}}c|}{$\SP$}
& \multicolumn{1}{>{\columncolor[gray]{0.95}}c|}{\begin{tabular}{c} $\SP$ \\ Thm.~\ref{DA2} \end{tabular}}\\

\hline
\end{tabular}
}
\caption{The existence of a $\Sigma_n \land \Pi_n$ sentence that is hereditarily doubly $(\Gamma, \Lambda)$-conservative over $T$}\label{table4}
\end{table}

Here, it is natural to ask the position of the condition `$T$ is a $\Sigma_n$-$(\SP, \Sigma_n)$ theory'.
Interestingly, we will show in Subsection \ref{SSec_impl} that in the $n=1$ case, the situation of the conditions is as illustrated in Figure \ref{fign}.
Thus, we see that condition `$T$ is a $\Sigma_1$-$(\Sigma_1{\downarrow}\Pi_1, \Sigma_1)$ theory' is strictly distinct from the prominent conditions concerning $\Pi_1$, $\Sigma_1$, and $\Sigma_1{\downarrow}\Pi_1$-conservativity.

\begin{figure}[ht]
\centering
\begin{tikzpicture}
\node (A) at (2,0) {$\Sigma_1{\downarrow}\Pi_1$-cons.};
\node (B) at (0,1.5) {$\Sigma_1$-$(\Sigma_1{\downarrow}\Pi_1, \Sigma_1)$};
\node (P) at (0,3) {$\Pi_{1}$-cons.};
\node (S) at (3,3) {$\Sigma_{1}$-cons.};

\draw [->, double] (P)--(B);
\draw [->, double] (B)--(A);
\draw [->, double] (S)--(A);

\draw [->, dashed] (P)--(S);
\draw [->, dashed] (S)--(B);
\draw [->, dashed] (B) to [out=135,in=225] (P);

\node at (1.5,3) {$\times$};
\node at (1.5,3)[above] {{\small Prop.~\ref{Prop_easy}}};
\node at (1.5,2.25) {$\times$};
\node at (1.9,1.9) {{\small Cor.~\ref{Cor_1}}};
\node at (-0.5,2.3) {$\times$};
\node at (-0.5,2.3)[left] {{\small Thm.~\ref{CE_1}}};

\end{tikzpicture}
\caption{The implications between the conditions for $n=1$}\label{fign}
\end{figure}
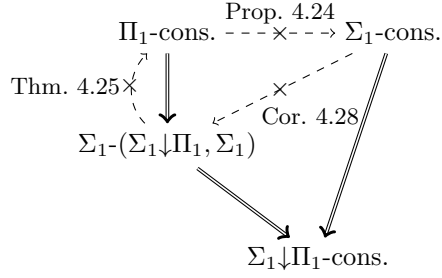

\subsection{$\mathcal{B}(\Sigma_n)$ sentences}

From Theorem \ref{DA2}, we also have that $\mathcal{B}(\Sigma_n) \cap \HDCons(\Sigma_n \land \Pi_n, \Delta_n; T) \neq \emptyset$ is equivalent to the $\SP$-conservativity of $T$ over $\PA$. 
As in the last subsection, for $\Theta = \mathcal{B}(\Sigma_n)$, the only remaining case is $\mathcal{B}(\Sigma_n) \cap \HDCons(\SP, \Sigma_n \land \Pi_n; T)$. 
We prove the following characterization for this case in a somewhat general form.

\begin{thm}\label{H*_Gamma_DA_S}
For $\Theta \in \{\mathcal{B}(\Sigma_n), \Delta_{n+1}(\PA), \Sigma_{n+1}, \Pi_{n+1}\}$, the following are equivalent: 
\begin{enumerate}
    \item $T$ is a $\Theta$-$(\SP, \Sigma_n)$ theory. 
    
    \item $\Theta \cap \HDCons(\SP, \Sigma_n \land \Pi_n; T) \neq \emptyset$. 

\end{enumerate}
\end{thm}
\begin{proof}
$(1 \Rightarrow 2)$: 
Suppose that $T$ is a $\Theta$-$(\SP, \Sigma_n)$ theory, that is, there exists a sentence $\varphi \in \Theta$ such that $T + \varphi$ is $\SP$-conservative over $\PA$ and $\neg \varphi \in \HCons(\Sigma_n, T)$. 
By Fact \ref{DA0}, we find a sentence $\psi \in \Pi_n \cap \HDCons(\Sigma_n, \Pi_n; T + \varphi)$. 
Whatever $\Theta$ is $\mathcal{B}(\Sigma_n)$, $\Delta_{n+1}(\PA)$, $\Sigma_{n+1}$, or $\Pi_{n+1}$, we obtain $\varphi \land \psi \in \Theta$.
So, it suffices to prove $\varphi \land \psi \in \HDCons(\SP, \Sigma_n \land \Pi_n; T)$. 
Since $\neg \varphi \in \HCons(\Sigma_n,T)$ and $\neg \psi \in \HCons(\Pi_n,T)$, $\neg (\varphi \land \psi) \in \HCons(\Sigma_n \land \Pi_n, T)$ is proved as in the proof of Theorem \ref{DA2}.

We show $\varphi \land \psi \in \HCons(\SP,T)$.
Let $U$ be any subtheory of $T$ and suppose $U + \varphi + \psi \vdash \sigma \land \pi$, where $\sigma \in \Sigma_n$ and $\pi \in \Pi_n$. 
Then, $U + \varphi + \psi \vdash \sigma \prec \neg \pi$. 
Since $\psi \in \HCons(\Sigma_n, T + \varphi)$, we have $U+ \varphi \vdash \sigma \prec \neg \pi$. 
Thus, $T+ \varphi \vdash \sigma \prec \neg \pi \land \neg(\neg \pi \preccurlyeq \sigma)$. 
Since $T+\varphi$ is $\SP$-conservative over $\PA$, we obtain $\PA \vdash \sigma \prec \neg \pi \vee \neg(\neg \pi \preccurlyeq \sigma)$, and hence $\PA \vdash \neg(\neg \pi \preccurlyeq \sigma)$ holds. 
It follows that $\PA + \neg \pi \vdash \sigma$, and we conclude $\PA \vdash \sigma \lor \pi$.

\medskip

$(2 \Rightarrow 1)$: Let $\varphi$ be a $\Theta$ sentence such that $\varphi \in  \HDCons(\SP, \Sigma_n \land \Pi_n; T)$. 
Let $\psi$ be a $\Sigma_n$ sentence satisfying 
\[
\PA \vdash \psi \leftrightarrow \PR_T^{\Sigma_n}(\gn{\neg(\varphi \vee \psi)}) \preccurlyeq \PR_T^{\Pi_n}(\gn{\varphi \vee \psi}). 
\]
In either case, the sentence $\varphi \vee \psi$ is in $\Theta$.

At first, we prove $\varphi \vee \psi \in \Cons(\Pi_n, T)$.
Let $\pi \in \Pi_n$ and suppose $T+ (\varphi \vee \psi) \vdash \pi$. 
Since $T+ \neg \pi \vdash \neg(\varphi \vee \psi)$, there exists a $q \in \omega$ such that $\PA + \neg \pi \vdash \Prf_T^{\Sigma_n}(\gn{\neg(\varphi \vee \psi)}, \num{q})$.
Since $T+ \neg(\varphi \vee \psi) \vdash \forall z < \num{q} \, \neg \Prf_T^{\Pi_n}(\gn{\varphi \vee \psi}, z)$, we obtain $T+ \neg \pi + \neg(\varphi \vee \psi) \vdash \psi$, and thus $T + \neg \pi \vdash \varphi \vee \psi$ holds. 
It follows that $T \vdash \pi$.

Secondly, we prove that $T+ (\varphi \vee \psi)$ is $\SP$-conservative over $\PA$. 
Let $\sigma \in \Sigma_n$ and $\pi \in \Pi_n$ be such that $T+ (\varphi \vee \psi )\vdash \sigma \wedge \pi$. 
Here, we prove $\PA + \varphi \vdash \sigma \vee \pi$.
Since $\varphi \vee \psi \in \Cons(\Pi_n,T)$, we obtain $T \vdash \pi$ and it follows that $\PA + (\varphi \vee \pi)$ is a subtheory of $T$.
Since $\PA+ (\varphi \vee \pi) +\neg \varphi \vdash \pi$ and $\neg \varphi \in \HCons(\Sigma_n \land \Pi_n,T)$, we obtain $\PA + (\varphi \vee \pi) \vdash \pi$. In particular, $\PA + \varphi \vdash \sigma \vee \pi$.
Next, we prove $\PA + \neg \varphi \vdash \sigma \vee \pi$.
Since $T \vdash \varphi \to \sigma \wedge \pi$, the theory $\PA + (\varphi \to \sigma \wedge \pi)$ is a subtheory of $T$. 
Since $\PA + (\varphi \to \sigma \wedge \pi) + \varphi \vdash \sigma \wedge \pi$ and $\varphi \in \HCons(\SP, T)$, we obtain
$\PA + (\varphi \to \sigma \wedge \pi) \vdash \sigma \vee \pi$.
It follows that $\PA + \neg \varphi \vdash \sigma \vee \pi$.
By the law of excluded middle, we conclude $\PA \vdash \sigma \vee \pi$.

Finally, we prove $\neg (\varphi \vee \psi ) \in \HCons (\Sigma_n,T)$. 
Let $U$ be any subtheory of $T$.
Suppose that $\sigma$ is a $\Sigma_n$ sentence and $U + \neg(\varphi \vee \psi) \vdash \sigma$.
Since $U + \neg \sigma \vdash \varphi \vee \psi$, there exists a natural number $p$ such that $\PA + \neg \sigma \vdash \Prf_T^{\Pi_n}(\gn{\varphi \vee \psi}, \num{p})$. 
Since $T+ (\varphi \vee \psi) \vdash \forall z \leq \num{p} \, \neg \Prf_T^{\Sigma_n}(\gn{\neg(\varphi \vee \psi)}, z)$, it follows that $T+ \neg \sigma + (\varphi \vee \psi) \vdash \neg \psi$, and hence $T + \neg \sigma \vdash \neg \psi$. 
Since $U+ \neg \sigma \vdash \varphi \vee \psi$, we have $T+ \neg \sigma \vdash \varphi$, and thus $T+ \neg \varphi \vdash \sigma$.
Since $\neg \varphi \in \HCons(\Sigma_n \wedge \Pi_n ,T)$, we obtain $T \vdash \sigma$. 

Since $T \vdash \varphi \vee \psi \to \forall z \leq \num{p} \, \neg \Prf_T^{\Sigma_n}(\gn{\neg(\varphi \vee \psi)}, z)$, we have
\begin{equation*}
T \vdash  \psi \to \forall z \leq \num{p} \, \neg \Prf_T^{\Sigma_n}(\gn{\neg(\varphi \vee \psi)}, z).
\end{equation*}
Thus, 
\[
    V : = \PA+ \Bigl(\sigma \wedge \bigl(\psi \to \forall z \leq \num{p} \, \neg \Prf_T^{\Sigma_n}(\gn{\neg(\varphi \vee \psi)}, z) \bigr)\Bigr) \lor \neg \varphi
\]
is a subtheory of $T$.
Since $V + \varphi$ proves the $\Sigma_n \land \Pi_n$ sentence $\sigma \wedge \bigl( \psi \to \forall z \leq \num{p} \, \neg \Prf_T^{\Sigma_n}(\gn{\neg(\varphi \vee \psi)}, z) \bigr)$, we obtain 
\[
    V \vdash \sigma \vee \bigl(\psi \to \forall z \leq \num{p} \, \neg \Prf_T^{\Sigma_n}(\gn{\neg(\varphi \vee \psi)}, z) \bigr)
\]
because $\varphi \in \HCons(\SP,T)$.
In particular, 
\[
    \PA + \neg \varphi \vdash \sigma \vee \bigl(\psi \to \forall z \leq \num{p} \, \neg \Prf_T^{\Sigma_n}(\gn{\neg(\varphi \vee \psi)}, z) \bigr).
\]
Since $U + \neg \sigma \vdash \varphi \vee \psi$, it follows that
\[
U + \neg \varphi + \neg \sigma \vdash \forall z \leq \num{p} \, \neg \Prf_T^{\Sigma_n}(\gn{\neg(\varphi \vee \psi)}, z).
\]
Since $\PA + \neg \sigma \vdash \Prf_T^{\Pi_n}(\gn{\varphi \vee \psi}, \num{p})$, we obtain $U + \neg \varphi + \neg \sigma \vdash \neg \psi$.
By combining this with $U + \neg \sigma + \neg \varphi \vdash \psi$, we have $U + \neg \varphi \vdash \sigma$.
Since $\neg \varphi \in \HCons(\Sigma_n \wedge \Pi_n,T)$, we conclude $U \vdash \sigma$.
\end{proof}

As in the last subsection, the condition `$T$ is a $\mathcal{B}(\Sigma_n)$-$(\SP, \Sigma_n)$ theory' also intermediates between $\SP$-conservativity and $\Pi_n$-conservativity (see also Figure \ref{Fig2}). 
Table \ref{table5} summarizes the $\Theta = \mathcal{B}(\Sigma_n)$ case. 
In the table, `$\mathcal{B}$-$(\downarrow, \Sigma)$' indicates that $T$ is a $\mathcal{B}(\Sigma_n)$-$(\SP, \Sigma_n)$ theory.

\begin{table}[ht]
\centering
\scriptsize{
\begin{tabular}{|c||c|c|c|c|c|}
\hline
\diagbox{$\Gamma$}{$\Lambda$} & $\Delta_n$ & $\SP$ & $\Sigma_n$ & $\Pi_n$ & $\Sigma_n \land \Pi_n$ \\
\hline

\hline 
$\Sigma_n \land \Pi_n$ 
& \multicolumn{1}{>{\columncolor[gray]{0.95}}c|}{\begin{tabular}{c} $\SP$ \\ Thm.~\ref{DA2} \end{tabular}}
& \multicolumn{1}{>{\columncolor[gray]{0.76}}c|}{\begin{tabular}{c} $\mathcal{B}$-$(\downarrow, \Sigma)$ \\ Thm.~\ref{H*_Gamma_DA_S} \end{tabular}}
& \multicolumn{1}{>{\columncolor[gray]{0.5}}c|}{$\times$}
& \multicolumn{1}{>{\columncolor[gray]{0.5}}c|}{$\times$}
& \multicolumn{1}{>{\columncolor[gray]{0.5}}c|}{$\times$} \\

\hline 
$\Pi_n$ & 
\multicolumn{1}{>{\columncolor[gray]{0.95}}c|}{$\SP$}
& \multicolumn{1}{>{\columncolor[gray]{0.95}}c|}{$\SP$}
& \multicolumn{1}{>{\columncolor[gray]{0.95}}c|}{\begin{tabular}{c} $\SP$ \\ Fact \ref{DA0} \end{tabular}}
& \multicolumn{1}{>{\columncolor[gray]{0.5}}c|}{\begin{tabular}{c} $\times$ \\ Fact \ref{Fact}.4 \\ Lindstr\"om \end{tabular}}
& \multicolumn{1}{>{\columncolor[gray]{0.5}}c|}{$\times$} \\

\hline 
$\Sigma_n$ & 
\multicolumn{1}{>{\columncolor[gray]{0.95}}c|}{$\SP$}
& \multicolumn{1}{>{\columncolor[gray]{0.95}}c|}{$\SP$}
& \multicolumn{1}{>{\columncolor[gray]{0.5}}c|}{\begin{tabular}{c} $\times$ \\ Fact \ref{Fact}.4 \\ Lindstr\"om \end{tabular}}
& \multicolumn{1}{>{\columncolor[gray]{0.95}}c|}{\begin{tabular}{c} $\SP$ \\ Fact \ref{DA0} \end{tabular}}
& \multicolumn{1}{>{\columncolor[gray]{0.5}}c|}{$\times$} \\

\hline 
$\SP$ 
& \multicolumn{1}{>{\columncolor[gray]{0.95}}c|}{$\SP$}
& \multicolumn{1}{>{\columncolor[gray]{0.95}}c|}{$\SP$}
& \multicolumn{1}{>{\columncolor[gray]{0.95}}c|}{$\SP$}
& \multicolumn{1}{>{\columncolor[gray]{0.95}}c|}{$\SP$}
& \multicolumn{1}{>{\columncolor[gray]{0.83}}c|}{\begin{tabular}{c} $\mathcal{B}$-$(\downarrow, \Sigma)$ \\ Thm.~\ref{H*_Gamma_DA_S} \end{tabular}} \\

\hline 
$\Delta_n$ 
& \multicolumn{1}{>{\columncolor[gray]{0.95}}c|}{\begin{tabular}{c} $\SP$ \\ Thm.~\ref{HDelta} \end{tabular}}
& \multicolumn{1}{>{\columncolor[gray]{0.95}}c|}{$\SP$}
& \multicolumn{1}{>{\columncolor[gray]{0.95}}c|}{$\SP$}
& \multicolumn{1}{>{\columncolor[gray]{0.95}}c|}{$\SP$}
& \multicolumn{1}{>{\columncolor[gray]{0.95}}c|}{\begin{tabular}{c} $\SP$ \\ Thm.~\ref{DA2} \end{tabular}}\\

\hline
\end{tabular}
}
\caption{The existence of a $\mathcal{B}(\Sigma_n)$ sentence that is hereditarily doubly $(\Gamma, \Lambda)$-conservative over $T$}\label{table5}
\end{table}

\subsection{$\Delta_{n+1}(\PA)$ sentences}

First, we have the following non-existence proposition. 

\begin{prop}\label{HDCons_S_S}
$\Delta_{n+1}(\PA) \cap \HDCons(\Sigma_n, \Sigma_n; T) = \emptyset$. 
\end{prop}
\begin{proof}
Suppose, towards a contradiction, that $\Delta_{n+1}(\PA) \ \cap  \ \HDCons(\Sigma_n, \Sigma_n; T) \neq \emptyset$. Then, we obtain $\Delta_{n+1}(\PA) \cap \DCons(\Sigma_n, \Sigma_n; \PA) \neq \emptyset$. Since $\PA$ is sound, this contradicts Fact \ref{Fact}.3.
\end{proof}

On the other hand, $\Delta_{n+1}(\PA) \cap \HDCons(\Pi_n, \Pi_n; T)$ may not be empty. 
Before proving our theorem, we prepare the following useful lemma. 

\begin{lem}\label{Lem_fp}
Suppose that $\Gamma$ and $\Lambda$ satisfy the so-called small reflection:
\begin{itemize}
    \item $T \vdash \Prf_T^\Gamma(\gn{\varphi}, \num{p}) \to \varphi$ for all sentences $\varphi$ and $p \in \omega$, 
    \item $T \vdash \Prf_T^\Lambda(\gn{\varphi}, \num{p}) \to \varphi$ for all sentences $\varphi$ and $p \in \omega$. 
\end{itemize}
Let $\xi$ be a sentence satisfying
\[
    T \vdash \xi \leftrightarrow \PR_T^\Gamma(\gn{\neg \xi}) \preccurlyeq \PR_T^{\Lambda}(\gn{\xi}). 
\]
Then, for all $k \in \omega$, we have $T \vdash \neg \Prf_T^{\Gamma}(\gn{\neg \xi}, \num{k})$ and $T \vdash \neg \Prf_T^{\Lambda}(\gn{\xi}, \num{k})$. 
\end{lem}
\begin{proof}
Suppose that $\Gamma$, $\Lambda$, and $\xi$ satisfy the required conditions. 
We obtain
\begin{align*}
T \vdash \Prf_{T}^{\Gamma}(\gn{\neg \xi}, \num{k}) & \to \Prf_{T}^{\Gamma}(\gn{\neg \xi},\num{k}) \wedge \neg \xi, \\
& \to \Prf_{T}^{\Gamma}(\gn{\neg \xi},\num{k}) \wedge \forall z \leq \num{k}\, \neg \Prf_{T}^{\Lambda}(\gn{\xi},z) \land \neg \xi, \\
& \to  \xi \wedge \neg \xi.  
\end{align*}
Thus, we obtain $T \vdash \neg \Prf_{T}^{\Gamma}(\gn{ \neg \xi}, \num{k})$.
Also, $T \vdash \neg \Prf_{T}^{\Lambda}(\gn{\xi}, \num{k})$ is proved in the same way. 
\end{proof}

\begin{thm}\label{Pi}
The following are equivalent: 
\begin{enumerate}
    \item $T$ is $\Pi_n$-conservative over $\PA$. 

    \item $T$ is a $\Delta_{n+1}(\PA)$-$(\Pi_n, \mathcal{B}(\Sigma_n))$ theory. 

    \item $\Delta_{n+1}(\PA) \cap \HDCons(\Pi_n, \mathcal{B}(\Sigma_n); T) \neq \emptyset$. 

    \item $\HDCons(\Pi_n, \Pi_n; T) \neq \emptyset$. 

\end{enumerate}
\end{thm}
\begin{proof}
$(1 \Rightarrow 2)$: Suppose $T$ is $\Pi_n$-conservative over $\PA$. Let $\varphi$ be a $\Sigma_{n+1}$ sentence satisfying
\[
\PA \vdash \varphi \leftrightarrow \PR_T^{\Sigma_n}(\gn{\neg \varphi}) \preccurlyeq \PR_T^{\Sigma_n \wedge \Pi_n}(\gn{\varphi}).
\]
Since
\[
\PA \vdash \varphi \leftrightarrow \PR_{T}^{\Sigma_n}(\gn{\neg \varphi}) \wedge \neg \bigl(\PR_{T}^{\Sigma_n \wedge \Pi_n}(\gn{\varphi}) \prec \PR_{T}^{\Sigma_n}(\gn{\neg \varphi}) \bigr),
\]
the sentence $\varphi$ is $\Delta_{n+1}(\PA)$.

We prove that $T + \varphi$ is $\Pi_n$-conservative over $\PA$. 
Let $\pi$ be a $\Pi_n$ sentence such that $T +\varphi \vdash \pi$. 
Since $T + \neg \pi \vdash \neg \varphi$, there exists $p \in \omega$ such that $\PA + \neg \pi \vdash \Prf_{T}^{\Sigma_n}(\gn{\neg \varphi}, \num{p})$. 
By Lemma \ref{Lem_fp}, $T \vdash \forall z < p \, \neg \Prf_{T}^{\Sigma_n \wedge \Pi_n}(\gn{\varphi}, z)$, and hence $T + \neg \pi \vdash \varphi$. 
Thus, we obtain $T \vdash \pi$. 
Since $T$ is $\Pi_n$-conservative over $\PA$, we have $\PA \vdash \pi$. 

We prove $\neg \varphi \in \HCons(\mathcal{B}(\Sigma_n),T)$. 
Let $U$ be any subtheory of $T$ and suppose $U + \neg \varphi \vdash \sigma \vee \pi$, where $\sigma \in \Sigma_n$ and $\pi \in \Pi_n$. 
Since $T + \neg \sigma + \neg \pi \vdash \varphi$, there exists $p \in \omega$ such that $\PA + \neg \sigma + \neg \pi \vdash \Prf_{T}^{\Sigma_n \wedge \Pi_n}(\gn{\varphi}, \num{p})$. 
By Lemma \ref{Lem_fp}, we obtain $T \vdash \forall z \leq \num{p} \, \neg \Prf_{T}^{\Sigma_n}(\gn{\varphi}, z)$. 
Since $T$ is $\Pi_n$-conservative over $\PA$, we have $\PA \vdash \forall z \leq \num{p} \, \neg \Prf_{T}^{\Sigma_n}(\gn{\varphi}, z)$.
Thus, we obtain $\PA + \neg \sigma + \neg \pi \vdash \neg \varphi$. 
Therefore, it follows that $U \vdash \sigma \vee \pi$.

\medskip

$(2 \Rightarrow 3)$: By Proposition \ref{Prop_triple}. 

\medskip

$(3 \Rightarrow 4)$: Trivial.

\medskip

$(4 \Rightarrow 1)$: By Proposition \ref{G_G}.
\end{proof}

In particular, Theorem \ref{Pi} states that the converse implication of Proposition \ref{Prop_triple} holds for the triple $(\Delta_{n+1}(\PA), \Pi_n, \mathcal{B}(\Sigma_n))$. 
We further prove that this is the case for other triples.

\begin{thm}\label{H*_Theta_Gamma_B}
For any $\Theta$ and $\Lambda \supseteq \mathcal{B}(\Sigma_n)$, the following are equivalent: 
\begin{enumerate}
    \item $T$ is a $\Theta$-$(\SP, \Lambda)$ theory. 

    \item $\Theta \cap \HDCons(\SP, \Lambda; T) \neq \emptyset$. 

\end{enumerate}
\end{thm}
\begin{proof}
$(1 \Rightarrow 2)$: By Proposition \ref{Prop_triple}. 

\medskip

$(2 \Rightarrow 1)$:
Let $\varphi$ be a $\Theta$ sentence such that $\varphi \in \HDCons(\SP, \Lambda;T)$.
It suffices to show that $T+ \varphi$ is $\SP$-conservative over $\PA$.

Let $\sigma \in \Sigma_n$ and $\pi \in \Pi_n$ be such that $T+ \varphi \vdash \sigma \land \pi$.
Since $\varphi \in \HCons(\SP, T)$, we obtain $T \vdash \sigma \vee \pi$. 
Then, we see that $\PA + (\sigma \vee \pi \lor \varphi)$ is a subtheory of $T$.
We have $\PA + (\sigma \vee \pi \vee \varphi) + \neg \varphi \vdash \sigma \vee \pi$. 
Since $\neg \varphi \in \HCons(\Lambda,T)$ and $\Lambda \supseteq \mathcal{B}(\Sigma_n)$, it follows that $\PA + (\sigma \vee \pi \vee \varphi) \vdash \sigma \vee \pi$. 
In particular, $\PA + \varphi \vdash \sigma \vee \pi$. 
On the other hand, it follows from $T \vdash \varphi \to \sigma \wedge \pi$ that $\PA + (\varphi \to \sigma \wedge \pi)$ is a subtheory of $T$. 
Since $\PA + (\varphi \to \sigma \wedge \pi) + \varphi \vdash \sigma \land \pi$ and $\varphi \in \HCons(\SP, T)$, we obtain $\PA + (\varphi \to \sigma \wedge \pi) \vdash \sigma \vee \pi$. 
Thus, $\PA + \neg \varphi \vdash \sigma \vee \pi$.
By the law of excluded middle, we conclude $\PA \vdash \sigma \vee \pi$.
\end{proof}

For $\Gamma = \Delta_n$, we make a little special care. 

\begin{defn}\label{Def_truple_Delta}
We say that a theory $T$ is a \textit{$\Theta$-$(\Delta_n, \Lambda)^*$ theory} iff there exists a sentence $\varphi$ satisfying the following three conditions: 
\begin{enumerate}
    \item $\varphi \in \Theta$. 

    \item For any subtheory $U$ of $T$, we have that $T + \varphi$ is $\Delta_n(U)$-conservative over $U$. 

    \item $\neg \varphi \in \HCons(\Lambda, T)$.
\end{enumerate}
\end{defn}

In Definition \ref{Def_truple_Delta}, if the second item is replaced by 
\begin{itemize}
    \item For any subtheory $U$ of $T + \varphi$, we have that $T + \varphi$ is $\Delta_n(U)$-conservative over $U$, 
\end{itemize}
then the resulting condition is equivalent to `$T$ is a $\Theta$-$(\SP, \Lambda)$ theory' by Theorem \ref{HDelta}. 
In fact, we will show in Subsection \ref{SSec_impl} that $\Delta_{2}(\PA)$-$(\Delta_1, \mathcal{B}(\Sigma_1))^*$ is strictly weaker than $\Delta_{2}(\PA)$-$(\Sigma_1{\downarrow}\Pi_1, \mathcal{B}(\Sigma_1))$ (Corollary \ref{Cor_triple2}). 
We then prove the following characterization. 

\begin{thm}\label{H*_Theta_Gamma_B_2}
For any $\Theta$ and $\Lambda \supseteq \SP$, the following are equivalent: 
\begin{enumerate}
    \item $T$ is a $\Theta$-$(\Delta_n, \Lambda)^*$ theory. 

    \item $\Theta \cap \HDCons(\Delta_n, \Lambda; T) \neq \emptyset$. 
\end{enumerate}
\end{thm}
\begin{proof}
$(1 \Rightarrow 2)$: 
Let $\varphi$ be a $\Theta$ sentence such that for any subtheory $U$ of $T$, we have that $T + \varphi$ is $\Delta_n(U)$-conservative over $U$, and $\neg \varphi \in \HCons(\Lambda,T)$.
We prove $\varphi \in \HDCons(\Delta_n, \Lambda;T)$. 
It suffices to prove $\varphi \in \HCons(\Delta_n ,T)$.
Let $U$ be any subtheory of $T$ and $\delta$ be a $\Delta_n(U)$ sentence such that $U + \varphi \vdash \delta$. 
Since $T + \varphi \vdash \delta$ and $T+ \varphi$ is $\Delta_n(U)$-conservative over $U$, we have that $U \vdash \delta$. 

\medskip

$(2 \Rightarrow 1)$:
Let $\varphi$ be a $\Theta$ sentence such that $\varphi \in \HDCons(\SP, \Lambda;T)$.
Let $U$ be any subtheory of $T$ and $\delta$ be a $\Delta_n(U)$ sentence. 
Suppose $T+ \varphi \vdash \delta$. 
We find $\sigma \in \Sigma_n$ and $\pi \in \Pi_n$ such that $U \vdash \delta \leftrightarrow \sigma$ and $U \vdash \delta \leftrightarrow \pi$. 
Since $\delta \in \Delta_n(U) \subseteq \Delta_n(T)$ and $\varphi \in \HCons(\Delta_n, T)$, we obtain $T \vdash \delta$, and hence $T \vdash \sigma \land \pi$. 
Then, $\PA + ((\sigma \land \pi) \lor \varphi)$ is a subtheory of $T$ and $\PA + ((\sigma \land \pi) \lor \varphi) + \neg \varphi \vdash \sigma \land \pi$. 
Since $\neg \varphi \in \HCons(\Lambda,T)$ and $\Lambda \supseteq \SP$, we obtain $\PA + ((\sigma \land \pi) \lor \varphi) \vdash \sigma \lor \pi$.
In particular, $\PA + \varphi \vdash \sigma \lor \pi$. 
Then, we have $U + \varphi \vdash \delta$.
Since $\varphi \in \HCons(\Delta_n,T)$, we conclude $U \vdash \delta$. 
\end{proof}

Table \ref{table6} summarizes the $\Theta = \Delta_{n+1}(\PA)$ case. 
In the table, `$\Delta_{n+1}$-$(\Delta, \mathcal{B})^*$', `$\Delta_{n+1}$-$(\downarrow, \mathcal{B})$', and `$\Delta_{n+1}$-$(\downarrow, \Sigma)$' indicate `$T$ is a $\Delta_{n+1}(\PA)$-$(\Delta_n, \mathcal{B}(\Sigma_n))^*$ theory', `$T$ is a $\Delta_{n+1}(\PA)$-$(\SP, \mathcal{B}(\Sigma_n))$ theory', and `$T$ is a $\Delta_{n+1}(\PA)$-$(\SP, \Sigma_n)$ theory', respectively. 

\begin{table}[ht]
\centering
\scriptsize{
\begin{tabular}{|c||c|c|c|c|c|c|}
\hline
\diagbox{$\Gamma$}{$\Lambda$} & $\Delta_n$ & $\SP$ & $\Sigma_n$ & $\Pi_n$ & $\Sigma_n \land \Pi_n$ & $\mathcal{B}(\Sigma_n)$\\
\hline

\hline 
$\mathcal{B}(\Sigma_n)$ 
& \multicolumn{1}{>{\columncolor[gray]{0.83}}c|}{\begin{tabular}{c} $\Delta_{n+1}$-$(\Delta, \mathcal{B})^*$ \\ Thm.~\ref{H*_Theta_Gamma_B_2} \end{tabular}}
& \multicolumn{1}{>{\columncolor[gray]{0.83}}c|}{\begin{tabular}{c} $\Delta_{n+1}$-$(\downarrow, \mathcal{B})$ \\ Thm.~\ref{H*_Theta_Gamma_B} \end{tabular}}
& \multicolumn{1}{>{\columncolor[gray]{0.5}}c|}{$\times$}
& \multicolumn{1}{>{\columncolor[gray]{0.75}}c|}{\begin{tabular}{c} $\Pi_n$ \\ Thm.~\ref{Pi} \end{tabular}}
& \multicolumn{1}{>{\columncolor[gray]{0.5}}c|}{$\times$}
& \multicolumn{1}{>{\columncolor[gray]{0.5}}c|}{$\times$} \\

\hline 
$\Sigma_n \land \Pi_n$ 
& \multicolumn{1}{>{\columncolor[gray]{0.95}}c|}{\begin{tabular}{c} $\SP$ \\ Thm.~\ref{DA2} \end{tabular}}
& \multicolumn{1}{>{\columncolor[gray]{0.83}}c|}{\begin{tabular}{c} $\Delta_{n+1}$-$(\downarrow, \Sigma)$ \\ Thm.~\ref{H*_Gamma_DA_S} \end{tabular}}
& \multicolumn{1}{>{\columncolor[gray]{0.5}}c|}{$\times$}
& \multicolumn{1}{>{\columncolor[gray]{0.75}}c|}{$\Pi_n$}
& \multicolumn{1}{>{\columncolor[gray]{0.5}}c|}{$\times$}
& \multicolumn{1}{>{\columncolor[gray]{0.5}}c|}{$\times$} \\

\hline 
$\Pi_n$ 
& \multicolumn{1}{>{\columncolor[gray]{0.95}}c|}{$\SP$}
& \multicolumn{1}{>{\columncolor[gray]{0.95}}c|}{$\SP$}
& \multicolumn{1}{>{\columncolor[gray]{0.95}}c|}{\begin{tabular}{c} $\SP$ \\ Fact \ref{DA0} \end{tabular}}
& \multicolumn{1}{>{\columncolor[gray]{0.75}}c|}{\begin{tabular}{c} $\Pi_n$ \\ Thm.~\ref{Pi} \end{tabular}}
& \multicolumn{1}{>{\columncolor[gray]{0.75}}c|}{$\Pi_n$}
& \multicolumn{1}{>{\columncolor[gray]{0.75}}c|}{\begin{tabular}{c} $\Pi_n$ \\ Thm.~\ref{Pi} \end{tabular}} \\

\hline 
$\Sigma_n$ 
& \multicolumn{1}{>{\columncolor[gray]{0.95}}c|}{$\SP$}
& \multicolumn{1}{>{\columncolor[gray]{0.95}}c|}{$\SP$}
& \multicolumn{1}{>{\columncolor[gray]{0.5}}c|}{\begin{tabular}{c} $\times$ \\ Prop.~\ref{HDCons_S_S} \end{tabular}}
& \multicolumn{1}{>{\columncolor[gray]{0.95}}c|}{\begin{tabular}{c} $\SP$ \\ Fact \ref{DA0} \end{tabular}}
& \multicolumn{1}{>{\columncolor[gray]{0.5}}c|}{$\times$}
& \multicolumn{1}{>{\columncolor[gray]{0.5}}c|}{$\times$} \\

\hline 
$\SP$ & 
\multicolumn{1}{>{\columncolor[gray]{0.95}}c|}{$\SP$}
& \multicolumn{1}{>{\columncolor[gray]{0.95}}c|}{$\SP$}
& \multicolumn{1}{>{\columncolor[gray]{0.95}}c|}{$\SP$}
& \multicolumn{1}{>{\columncolor[gray]{0.95}}c|}{$\SP$}
& \multicolumn{1}{>{\columncolor[gray]{0.83}}c|}{\begin{tabular}{c} $\Delta_{n+1}$-$(\downarrow, \Sigma)$ \\ Thm.~\ref{H*_Gamma_DA_S} \end{tabular}}
& \multicolumn{1}{>{\columncolor[gray]{0.83}}c|}{\begin{tabular}{c} $\Delta_{n+1}$-$(\downarrow, \mathcal{B})$ \\ Thm.~\ref{H*_Theta_Gamma_B} \end{tabular}} \\

\hline 
$\Delta_n$ 
& \multicolumn{1}{>{\columncolor[gray]{0.95}}c|}{\begin{tabular}{c} $\SP$ \\ Thm.~\ref{HDelta} \end{tabular}}
& \multicolumn{1}{>{\columncolor[gray]{0.95}}c|}{$\SP$}
& \multicolumn{1}{>{\columncolor[gray]{0.95}}c|}{$\SP$}
& \multicolumn{1}{>{\columncolor[gray]{0.95}}c|}{$\SP$}
& \multicolumn{1}{>{\columncolor[gray]{0.95}}c|}{\begin{tabular}{c} $\SP$ \\ Thm.~\ref{DA2} \end{tabular}}
& \multicolumn{1}{>{\columncolor[gray]{0.83}}c|}{\begin{tabular}{c} $\Delta_{n+1}$-$(\Delta, \mathcal{B})^*$ \\ Thm.~\ref{H*_Theta_Gamma_B_2} \end{tabular}}\\

\hline
\end{tabular}
}
\caption{The existence of a $\Delta_{n+1}(\PA)$ sentence that is hereditarily doubly $(\Gamma, \Lambda)$-conservative over $T$}\label{table6}
\end{table}

\subsection{$\Sigma_{n+1}$ sentences}

We prove that in the $\Theta = \Sigma_{n+1}$ case, $\SP$-conservativity is also a necessary and sufficient condition for the existence of certain cases. 
In this subsection, we use the $\Delta_n$-relativized proof predicates $\Prf_T^{(\Delta_n,\Sigma_n)}(x,y)$ and
$\Prf_T^{(\Delta_n,\Pi_n)}(x,y)$ introduced in Section \ref{Sec:Pre}.
Before proving our theorem, we prepare the following lemma. 

\begin{lem} \label{lem2}
Let $\varphi$ be any sentence. 
\begin{enumerate}
    \item For any $p \in \omega$, there exists $q\in \omega$ such that $T \vdash \Prf_T^{(\Delta_n, \Sigma_n)}(\gn{\varphi}, \num{p}) \to  \Prf_T^{(\Delta_n, \Pi_n)}(\gn{\varphi}, \num{q})$. 

    \item For any $p \in \omega$, there exists $q \in \omega$ such that $T \vdash \Prf_T^{(\Delta_n, \Pi_n)}(\gn{\varphi}, \num{p}) \to  \Prf_T^{(\Delta_n, \Sigma_n)}(\gn{\varphi}, \num{q})$. 

\end{enumerate}
\end{lem}
\begin{proof}
We give only a proof of Clause 1. 
Let $\varphi$ be any sentence and let $p$ be any natural number. 
We distinguish the following two cases.

\paragraph{Case 1:}
There exist $\xi \in \Sigma_n$, $\nu \in \Pi_n$, and $r \leq p$ such that $p$ is a $T$-proof of $\xi \to \varphi$ and $r$ is a $T$ proof of $\xi \leftrightarrow \nu$. 

In this case, there is a $T$-proof $q \geq p$ of $\nu \to \varphi$.
We obtain
\begin{align*}
T \vdash \xi & \to \num{r} \leq \num{q} \wedge \Prf_T(\gn{\xi \leftrightarrow \nu}, \num{r}) \wedge \nu \wedge \Prf_T(\gn{\nu \to \varphi}, \num{q}), \\
& \to \num{r} \leq \num{q} \wedge \Prf_T(\gn{\xi \leftrightarrow \nu}, \num{r}) \wedge \True_{\Pi_n}(\gn{\nu}) \wedge \Prf_T(\gn{\nu \to \varphi}, \num{q}), \\
& \to \Prf_T^{(\Delta_n,\Pi_n)}(\gn{\varphi},\num{q}).
\end{align*}
Since 
\begin{align*}
\PA \vdash \True_{\Sigma_n}(u) \wedge \Prf_T(\gn{u \dot{\to} \varphi},\num{p}) & \to \True_{\Sigma_n}(u) \wedge u= \gn{\xi}, \\
& \to \xi,
\end{align*}
we have $\PA \vdash \Prf_T^{(\Delta_n,\Sigma_n)}(\gn{\varphi},\num{p}) \to \xi$.
Thus, it follows that
\[
    T \vdash \Prf_T^{(\Delta_n, \Sigma_n)}(\gn{\varphi}, \num{p}) \to  \Prf_T^{(\Delta_n, \Pi_n)}(\gn{\varphi}, \num{q}).
\]

\paragraph{Case 2:} Otherwise.

In this case, we have
\[
    \PA \vdash \forall u, v, w \leq \num{p} \, \bigl(\Sigma_n(u) \wedge \Pi_n(v) \wedge \Prf_{T}(u \dot{\leftrightarrow} v, w) \to \neg \Prf_{T}(u \dot{\to} x, y)\bigr).
\]
In particular, $\PA \vdash \neg \Prf_T^{(\Delta_n,\Sigma_n)}(\gn{\varphi}, \num{p})$, and hence the lemma trivially holds. 
\end{proof}

\begin{thm}\label{DA3}
The following are equivalent: 
\begin{enumerate}
    \item $T$ is $\SP$-conservative over $\PA$. 

   \item $\Sigma_{n+1} \cap \HDCons(\Delta_n, \mathcal{B}(\Sigma_n); T) \neq \emptyset$. 

    \item $\Sigma_{n+1} \cap \HDCons(\mathcal{B}(\Sigma_n), \Delta_n; T) \neq \emptyset$. 
\end{enumerate}
\end{thm}
\begin{proof}
$(1 \Rightarrow 2)$:
Suppose that $T$ is $\SP$-conservative over $\PA$.
Let $\varphi$ be a $\Sigma_{n+1}$ sentence satisfying
\[
\PA \vdash \varphi \leftrightarrow \bigl(\PR_T^{(\Delta_n, \Sigma_n)}(\gn{\neg \varphi}) \preccurlyeq \PR_T^{\Sigma_n \wedge \Pi_n}(\gn{\varphi}) \bigr) \wedge \bigl(\PR_T^{(\Delta_n, \Pi_n)}(\gn{\neg \varphi}) \preccurlyeq \PR_T^{\Sigma_n \wedge \Pi_n}(\gn{\varphi}) \bigr).
\]

First, we prove $\varphi \in \HCons(\Delta_n,T)$. 
Let $U$ be any subtheory of $T$ and $\delta$ be a $\Delta_n(U)$ sentence such that $U + \varphi \vdash \delta$.
We find $\sigma \in \Sigma_n$ and $\pi \in \Pi_n$ such that $U \vdash \delta \leftrightarrow \sigma$ and $U \vdash \delta \leftrightarrow \pi$.
Since $U + \neg \delta \vdash \neg \varphi$, there exist natural numbers $p, q \in \omega$ such that
\[
    U + \neg \delta \vdash \Prf_T^{(\Delta_n , \Sigma_n)}(\gn{\neg \varphi}, \num{p}) \wedge \Prf_T^{(\Delta_n, \Pi_n)}(\gn{\neg \varphi}, \num{q}).
\]
Since 
\[
    T + \neg \varphi \vdash \forall z \leq \num{p} \, \neg \Prf_T^{\Sigma_n \wedge \Pi_n}(\gn{\varphi},z) \wedge \forall z \leq \num{q} \, \neg \Prf_T^{\Sigma_n \wedge \Pi_n}(\gn{\varphi},z),
\]
we have $T + \neg \varphi + \neg \delta \vdash \varphi$, and so $T + \neg \delta \vdash \varphi$. 
By combining this with $U + \varphi \vdash \delta$, we obtain $T \vdash \delta$. 
Then, $T \vdash \sigma \land \pi$. 
Since $T$ is $\SP$-conservative over $\PA$, we obtain $\PA \vdash \sigma \lor \pi$. Therefore, $U \vdash \delta$ holds.

Second, we prove $\neg \varphi \in \HCons(\mathcal{B}(\Sigma_n),T)$. 
Let $U$ be any subtheory of $T$.
Let $\sigma \in \Sigma_n$ and $\pi \in \Pi_n$ be such that $U + \neg \varphi \vdash \sigma \vee \pi$. 
Since $T + \neg \pi + \neg \sigma \vdash \varphi$, there exists a natural number $p$ such that $\PA + \neg \sigma + \neg \pi \vdash \Prf_T^{\Sigma_n \wedge \Pi_n}(\gn{\varphi}, \num{p})$. 
For each natural number $k \leq p$, by Lemma \ref{lem2} we find a natural number $q $ such that $T \vdash \Prf_T^{(\Delta_n, \Sigma_n)}(\gn{\neg \varphi}, \num{k}) \to \Prf_T^{(\Delta_n, \Pi_n)}(\gn{\neg \varphi}, \num{q})$.
Since
\begin{align*}
T \vdash \Prf_T^{(\Delta_n, \Sigma_n)}(\gn{\neg \varphi}, \num{k})
& \to \Prf_T^{(\Delta_n, \Sigma_n)}(\gn{\neg \varphi}, \num{k}) \land \neg \varphi, \\
& \to \Prf_T^{(\Delta_n, \Sigma_n)}(\gn{\neg \varphi}, \num{k}) \land \forall z \leq \num{k}  \  \neg \Prf_T^{\Sigma_n \wedge \Pi_n}(\gn{\varphi},z) \\
& \quad \quad \wedge \Prf_T^{(\Delta_n, \Pi_n)}(\gn{\neg \varphi}, \num{q}) \land \forall z \leq \num{q}  \  \neg \Prf_T^{\Sigma_n \wedge \Pi_n}(\gn{\varphi},z),\\
& \to \varphi,
\end{align*}
we obtain
\begin{align*}
T \vdash \neg \Prf_T^{(\Delta_n, \Sigma_n)}(\gn{\neg \varphi}, \num{k}).
\end{align*}
The same argument yields
$T \vdash \neg \Prf_T^{(\Delta_n, \Pi_n)}(\gn{\neg \varphi}, \num{k})$, and hence
\[
T \vdash \forall z \leq \num{p} \ \neg \Prf_T^{(\Delta_n, \Sigma_n)}(\gn{\neg \varphi}, z) \wedge \forall z \leq \num{p} \ \neg  \Prf_T^{(\Delta_n, \Pi_n)}(\gn{\neg \varphi}, z).
\]
Since $T$ is $\SP$-conservative over $\PA$, we obtain \[
\PA \vdash \forall z \leq \num{p} \ \neg \Prf_T^{(\Delta_n, \Sigma_n)}(\gn{\neg \varphi}, z)  \vee  \forall z \leq \num{p} \ \neg  \Prf_T^{(\Delta_n, \Pi_n)}(\gn{\neg \varphi}, z).
\]
By combining this with $\PA + \neg \sigma + \neg \pi \vdash \Prf_T^{\Sigma_n \wedge \Pi_n}(\gn{\neg \varphi}, \num{p})$, we have
\[
\PA + \neg \sigma + \neg \pi \vdash \bigl(\PR_T^{\Sigma_n \wedge \Pi_n}(\gn{\varphi}) \prec \PR_T^{(\Delta_n, \Sigma_n)}(\gn{\neg \varphi})\bigr) \lor \bigl(\PR_T^{\Sigma_n \wedge \Pi_n}(\gn{\varphi}) \prec \PR_T^{(\Delta_n, \Pi_n)}(\gn{\neg \varphi}) \bigr).
\]
It follows that $\PA + \neg \sigma + \neg \pi \vdash \neg \varphi$.
Since $U + \neg \varphi \vdash \sigma \vee \pi$, we conclude $U \vdash \sigma \vee \pi$.

\medskip

$(1 \Rightarrow 3)$: Let $\varphi$ be a $\Sigma_{n+1}$ sentence satisfying
\[
\PA \vdash \varphi \leftrightarrow \bigl(\PR_T^{\Sigma_n \wedge \Pi_n}(\gn{\neg \varphi}) \preccurlyeq \PR_T^{(\Delta_n, \Sigma_n)}(\gn{\varphi}) \bigr) \vee \bigl(\PR_T^{\Sigma_n \wedge \Pi_n}(\gn{\neg \varphi}) \preccurlyeq \PR_T^{(\Delta_n, \Pi_n)}(\gn{\varphi}) \bigr).
\]
Then $\varphi \in \HDCons(\mathcal{B}(\Sigma_n), \Delta_n;T)$ is proved in the same way as in the proof of the implication $(1 \Rightarrow 2)$.

\medskip

$(2 \Rightarrow 1)$ and $(3 \Rightarrow 1)$: 
These implications follow from Theorem \ref{HDelta}. 
\end{proof}

\begin{thm}\label{Sigma}
Let $\Gamma \in \{\Sigma_n, \Sigma_n \land \Pi_n, \mathcal{B}(\Sigma_n)\}$. 
The following are equivalent: 
\begin{enumerate}
    \item $T$ is $\Gamma$-conservative over $\PA$. 

    \item $T$ is a $\Sigma_{n+1}$-$(\Gamma, \mathcal{B}(\Sigma_n))$ theory.

    \item $T$ is a $\Sigma_{n+1}$-$(\mathcal{B}(\Sigma_n), \Gamma)$ theory.

    \item $\Sigma_{n+1} \cap \HDCons(\Gamma, \mathcal{B}(\Sigma_n); T) \neq \emptyset$. 

    \item $\Sigma_{n+1} \cap \HDCons(\mathcal{B}(\Sigma_n), \Gamma; T) \neq \emptyset$. 

    \item $\HDCons(\Gamma, \Gamma; T) \neq \emptyset$. 

\end{enumerate}
\end{thm}
\begin{proof}
$(1 \Rightarrow 2)$:
Suppose that $T$ is $\Gamma$-conservative over $\PA$. 
We give proofs depending on $\Gamma$. 

$\Gamma = \Sigma_n$: Let $\varphi$ be a $\Sigma_{n+1}$ sentence satisfying
\[
\PA \vdash \varphi \leftrightarrow \PR_T^{\Pi_n}(\gn{\neg \varphi}) \preccurlyeq \PR_T^{\Sigma_n \wedge \Pi_n}(\gn{\varphi}).
\]
Let $\sigma \in \Sigma_n$ be such that $T + \varphi \vdash \sigma$.
As in the same argument so far, $T \vdash \sigma$ is proved.
Since $T$ is $\Sigma_n$-conservative over $\PA$, we obtain $\PA \vdash \sigma$.
We have proved that $T + \varphi$ is $\Sigma_n$-conservative over $\PA$.

We prove $\neg \varphi \in \HCons(\mathcal{B}(\Sigma_n),T)$. 
Let $U$ be any subtheory of $T$. 
Let $\sigma \in \Sigma_n$ and $\pi \in \Pi_n$ be such that $U+ \neg \varphi \vdash \sigma \vee \pi$.
Then, there exists a natural number $p$ such that $\PA + \neg \sigma + \neg \pi \vdash \Prf_T^{\Sigma_n \wedge \Pi_n}(\gn{\varphi}, \num{p})$.
By Lemma \ref{Lem_fp}, we obtain 
$T \vdash \forall z \leq \num{p} \ \neg \Prf_T^{\Pi_n}(\gn{\neg \varphi})$. 
Since $T$ is $\Sigma_n$-conservative over $\PA$, we obtain $\PA \vdash \forall z \leq \num{p} \ \neg \Prf_T^{\Pi_n}(\gn{\neg \varphi})$.
Thus, $\PA + \neg \sigma + \neg \pi \vdash \neg \varphi$, and it follows that $U \vdash \sigma \vee \pi$.

\medskip

$\Gamma = \Sigma_n \wedge \Pi_n$: Let $\varphi$ be a $\Sigma_{n+1}$ sentence satisfying
\[
\PA \vdash \varphi \leftrightarrow \bigl(\PR_T^{\Pi_n}(\gn{\neg \varphi}) \vee \PR_T^{\Sigma_n}(\gn{\neg \varphi}) \bigr) \preccurlyeq \PR_T^{\Sigma_n \wedge \Pi_n}(\gn{\varphi}).
\]
We prove that $T + \varphi$ is $\Sigma_n \wedge \Pi_n$-conservative over $\PA$.
Let $\sigma \in \Sigma_n$ and $\pi \in \Pi_n$ be such that $T + \varphi \vdash \sigma \wedge \pi$. 
We have $T + (\neg \sigma \vee \neg \pi) \vdash \neg \varphi$, and hence $T + \neg \sigma  \vdash \neg \varphi$ and $T + \neg \pi \vdash \neg \varphi$. 
Then, $T + (\neg \sigma \lor \neg \pi) \vdash \varphi$ is proved in the usual way. 
So, $T \vdash \sigma \wedge \pi$. 
Since $T$ is $\Sigma_n \wedge \Pi_n$-conservative over $\PA$, we have $\PA \vdash \sigma \wedge \pi$.

We prove $\neg \varphi \in \HCons(\mathcal{B}(\Sigma_n),T)$. 
In the similar argument as we have proved Lemma \ref{Lem_fp}, it is proved that for all $p \in \omega$, $T \vdash \forall z \leq \num{p} \, (\neg \Prf_T^{\Sigma_n}\bigl(\gn{\neg \varphi}, z) \wedge \neg \Prf_T^{\Pi_n}(\gn{\neg \varphi}, z) \bigr)$.
Thus, $\neg \varphi \in \HCons(\mathcal{B}(\Sigma_n),T)$ is proved in the same way as the case $\Gamma = \Sigma_n$.

\medskip

$\Gamma = \mathcal{B}(\Sigma_n)$: 
Let $\varphi$ be a $\Sigma_{n+1}$ sentence satisfying
\[
\PA \vdash \varphi \leftrightarrow \PR_T^{\Sigma_n \wedge \Pi_n}(\gn{\neg \varphi}) \preccurlyeq \PR_T^{\Sigma_n \wedge \Pi_n}(\gn{\varphi}).
\]
It is proved that $T + \varphi$ is $\mathcal{B}(\Sigma_n)$-conservative over $\PA$ and $\neg \varphi \in \HCons(\mathcal{B}(\Sigma_n), T)$ in the same way as the case $\Gamma = \Sigma_n$.

\medskip

$(1 \Rightarrow 3)$: Suppose that $T$ is $\Gamma$-conservative over $\PA$. 
The case $\Gamma = \mathcal{B}(\Sigma_n)$ is already proved in the proof of $(1 \Rightarrow 2)$. 
We consider the following two cases.
\begin{itemize}
\item $\Gamma = \Sigma_n$: Let $\varphi$ be a $\Sigma_{n+1}$ sentence satisfying
\[
\PA \vdash \varphi \leftrightarrow \PR_T^{\Sigma_n \wedge \Pi_n}(\gn{\neg \varphi}) \preccurlyeq \PR_T^{\Pi_n}(\gn{\varphi}).
\]
\item $\Gamma = \Sigma_n \wedge \Pi_n$: Let $\varphi$ be a $\Sigma_{n+1}$ sentence satisfying
\[
\PA \vdash \varphi \leftrightarrow \PR_T^{\Sigma_n \wedge \Pi_n}(\gn{\neg \varphi}) \preccurlyeq \bigl(\PR_T^{\Pi_n}(\gn{\varphi}) \vee  \PR_T^{\Sigma_n}(\gn{\varphi}) \bigr).
\]
\end{itemize}
In the same way as our proof of $(1 \Rightarrow 2)$, it is proved that $\varphi$ is a witness of the property that $T$ is a $\Sigma_{n+1}$-$(\mathcal{B}(\Sigma_n),\Gamma)$ theory.

\medskip

$(2 \Rightarrow 4)$ and $(3 \Rightarrow 5)$: By Proposition \ref{Prop_triple}.

\medskip

$(4 \Rightarrow 6)$ and $(5 \Rightarrow 6)$: Straightforward. 

\medskip

$(6 \Rightarrow 1)$: By Proposition \ref{G_G}.
\end{proof}

\begin{table}[ht]
\centering
\scriptsize{
\begin{tabular}{|c||c|c|c|c|c|c|}
\hline
\diagbox{$\Gamma$}{$\Lambda$} & $\Delta_n$ & $\SP$ & $\Sigma_n$ & $\Pi_n$ & $\Sigma_n \land \Pi_n$ & $\mathcal{B}(\Sigma_n)$\\
\hline

\hline 
$\mathcal{B}(\Sigma_n)$ 
& \multicolumn{1}{>{\columncolor[gray]{0.95}}c|}{\begin{tabular}{c} $\SP$ \\ Thm.~\ref{DA3} \end{tabular}}
& \multicolumn{1}{>{\columncolor[gray]{0.83}}c|}{\begin{tabular}{c} $\Pi_{n+1}$-$(\downarrow, \mathcal{B})$ \\ Thm.~\ref{H*_Theta_Gamma_B} \end{tabular}}
& \multicolumn{1}{>{\columncolor[gray]{0.75}}c|}{\begin{tabular}{c} $\Sigma_n$ \\ Thm.~\ref{Sigma} \end{tabular}}
& \multicolumn{1}{>{\columncolor[gray]{0.75}}c|}{\begin{tabular}{c} $\Pi_n$ \\ Thm.~\ref{Pi} \end{tabular}}
& \multicolumn{1}{>{\columncolor[gray]{0.65}}c|}{\begin{tabular}{c} $\Sigma_n \land \Pi_n$ \\ Thm.~\ref{Sigma} \end{tabular}}
& \multicolumn{1}{>{\columncolor[gray]{0.6}}c|}{\begin{tabular}{c} $\mathcal{B}(\Sigma_n)$ \\ Thm.~\ref{Sigma} \end{tabular}} \\

\hline 
$\Sigma_n \land \Pi_n$ 
& \multicolumn{1}{>{\columncolor[gray]{0.95}}c|}{$\SP$}
& \multicolumn{1}{>{\columncolor[gray]{0.83}}c|}{\begin{tabular}{c} $\Pi_{n+1}$-$(\downarrow, \Sigma)$ \\ Thm.~\ref{H*_Gamma_DA_S} \end{tabular}}
& \multicolumn{1}{>{\columncolor[gray]{0.75}}c|}{$\Sigma_n$}
& \multicolumn{1}{>{\columncolor[gray]{0.75}}c|}{$\Pi_n$}
& \multicolumn{1}{>{\columncolor[gray]{0.65}}c|}{\begin{tabular}{c} $\Sigma_n \land \Pi_n$ \\ Thm.~\ref{Sigma} \end{tabular}}
& \multicolumn{1}{>{\columncolor[gray]{0.65}}c|}{\begin{tabular}{c} $\Sigma_n \land \Pi_n$ \\ Thm.~\ref{Sigma} \end{tabular}} \\

\hline 
$\Pi_n$ 
& \multicolumn{1}{>{\columncolor[gray]{0.95}}c|}{$\SP$}
& \multicolumn{1}{>{\columncolor[gray]{0.95}}c|}{$\SP$}
& \multicolumn{1}{>{\columncolor[gray]{0.95}}c|}{\begin{tabular}{c} $\SP$ \\ Fact \ref{DA0} \end{tabular}}
& \multicolumn{1}{>{\columncolor[gray]{0.75}}c|}{\begin{tabular}{c} $\Pi_n$ \\ Thm.~\ref{Pi} \end{tabular}}
& \multicolumn{1}{>{\columncolor[gray]{0.75}}c|}{$\Pi_n$}
& \multicolumn{1}{>{\columncolor[gray]{0.75}}c|}{\begin{tabular}{c} $\Pi_n$ \\ Thm.~\ref{Pi} \end{tabular}} \\

\hline 
$\Sigma_n$ 
& \multicolumn{1}{>{\columncolor[gray]{0.95}}c|}{$\SP$}
& \multicolumn{1}{>{\columncolor[gray]{0.95}}c|}{$\SP$}
& \multicolumn{1}{>{\columncolor[gray]{0.75}}c|}{\begin{tabular}{c} $\Sigma_n$ \\ Thm.~\ref{Sigma} \end{tabular}}
& \multicolumn{1}{>{\columncolor[gray]{0.95}}c|}{\begin{tabular}{c} $\SP$ \\ Fact \ref{DA0} \end{tabular}}
& \multicolumn{1}{>{\columncolor[gray]{0.75}}c|}{$\Sigma_n$}
& \multicolumn{1}{>{\columncolor[gray]{0.75}}c|}{\begin{tabular}{c} $\Sigma_n$ \\ Thm.~\ref{Sigma} \end{tabular}} \\

\hline 
$\SP$ & 
\multicolumn{1}{>{\columncolor[gray]{0.95}}c|}{$\SP$}
& \multicolumn{1}{>{\columncolor[gray]{0.95}}c|}{$\SP$}
& \multicolumn{1}{>{\columncolor[gray]{0.95}}c|}{$\SP$}
& \multicolumn{1}{>{\columncolor[gray]{0.95}}c|}{$\SP$}
& \multicolumn{1}{>{\columncolor[gray]{0.83}}c|}{\begin{tabular}{c} $\Sigma_{n+1}$-$(\downarrow, \Sigma)$ \\ Thm.~\ref{H*_Gamma_DA_S} \end{tabular}}
& \multicolumn{1}{>{\columncolor[gray]{0.83}}c|}{\begin{tabular}{c} $\Sigma_{n+1}$-$(\downarrow, \mathcal{B})$ \\ Thm.~\ref{H*_Theta_Gamma_B} \end{tabular}} \\

\hline 
$\Delta_n$ 
& \multicolumn{1}{>{\columncolor[gray]{0.95}}c|}{\begin{tabular}{c} $\SP$ \\ Thm.~\ref{HDelta} \end{tabular}}
& \multicolumn{1}{>{\columncolor[gray]{0.95}}c|}{$\SP$}
& \multicolumn{1}{>{\columncolor[gray]{0.95}}c|}{$\SP$}
& \multicolumn{1}{>{\columncolor[gray]{0.95}}c|}{$\SP$}
& \multicolumn{1}{>{\columncolor[gray]{0.95}}c|}{$\SP$}
& \multicolumn{1}{>{\columncolor[gray]{0.95}}c|}{\begin{tabular}{c} $\SP$ \\ Thm.~\ref{DA3} \end{tabular}}\\

\hline
\end{tabular}
}
\caption{The existence of a $\Sigma_{n+1}$ sentence that is hereditarily doubly $(\Gamma, \Lambda)$-conservative over $T$}\label{table7}
\end{table}

\subsection{Implications and non-implications between the conditions}\label{SSec_impl}

In this section, we examine implications and non-implications between the conditions we have addressed. 
Figure \ref{Fig2} sums up implications between the conditions. 
Almost all of the implications in the figure are trivial, but the following non-trivial implication holds when $n=1$.

\begin{prop}\label{S1-sound}
If $T$ is $\Sigma_1$-conservative over $\PA$, then $T$ is a $\Delta_2(\PA)$-$(\Delta_1, \mathcal{B}(\Sigma_1))^*$ theory. 
\end{prop}
\begin{proof}
Suppose that $T$ is $\Sigma_1$-conservative over $\PA$. 
Let $\varphi$ be a $\Sigma_2$ sentence satisfying the following equivalence:
\[
    \PA \vdash \varphi \leftrightarrow \PR_T(\gn{\neg \varphi}) \preccurlyeq \PR_T^{\Sigma_1 \land \Pi_1}(\gn{\varphi}). 
\]
It is shown that $\varphi \in \Delta_2(\PA)$ and $T \nvdash \neg \varphi$. 
Since $\PA \vdash \forall y \leq \num{p}\, \neg \Prf_T(\gn{\varphi}, y)$ for all $p \in \omega$, it is proved that $\neg \varphi \in \HCons(\mathcal{B}(\Sigma_1), T)$. 

We prove $\varphi \in \HCons(\Delta_1, T)$. 
Let $U$ be any subtheory of $T$ and $\delta \in \Delta_1(U)$ be such that $U +  \varphi \vdash \delta$. 
We find $\sigma \in \Sigma_1$ and $\pi \in \Pi_1$ such that $U \vdash (\delta \leftrightarrow \sigma) \land (\delta \leftrightarrow \pi)$. 
We have $T \vdash \sigma \lor \neg \pi$ because $T \vdash \delta \lor \neg \delta$. 
By the $\Sigma_1$-conservativity of $T$ over $\PA$, we have $\PA \vdash \sigma \lor \neg \pi$. 
Then, either $\mathbb{N} \models \sigma$ or $\mathbb{N} \models \neg \pi$. 
It follows from $T \nvdash \neg \varphi$ that $T \nvdash \neg \delta$ and $T \nvdash \neg \pi$. 
Since $\neg \pi$ is $\Sigma_1$, by $\Sigma_1$-completeness, we have $\mathbb{N} \models \pi$. 
Hence $\mathbb{N} \models \sigma$. 
By $\Sigma_1$-completeness again, $U \vdash \sigma$, and thus $U \vdash \delta$. 

We have proved $\Delta_2(\PA) \cap \HDCons(\Delta_1, \mathcal{B}(\Sigma_1); T) \neq \emptyset$. 
By Theorem \ref{H*_Theta_Gamma_B_2}, we conclude that $T$ is a $\Delta_2(\PA)$-$(\Delta_1, \mathcal{B}(\Sigma_1))^*$ theory.  
\end{proof}

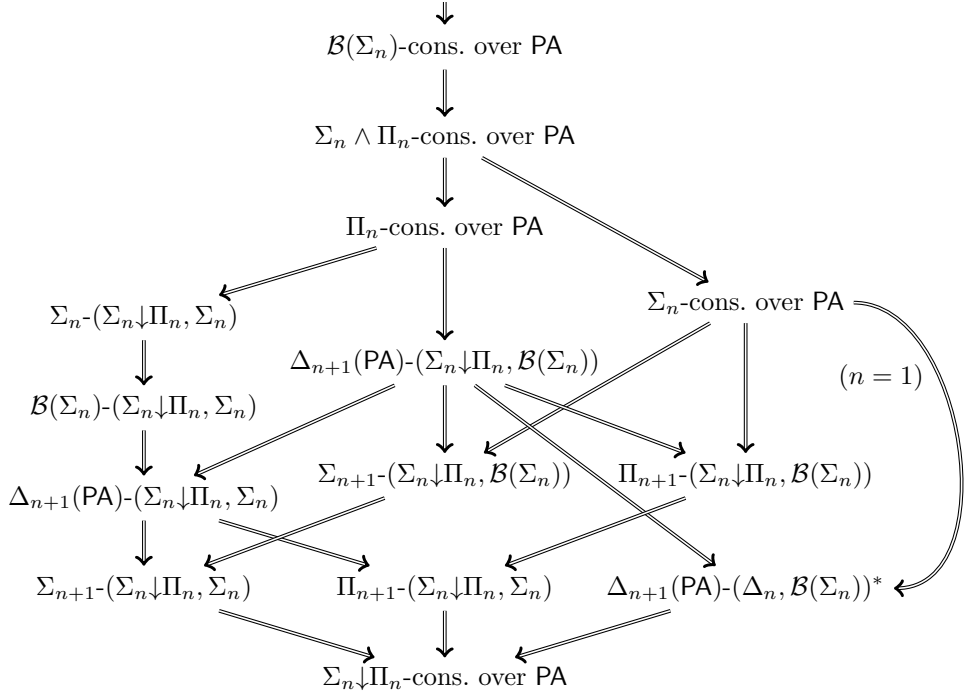
\begin{figure}[ht]
\centering
\begin{tikzpicture}

\node (B) at (0, 8.4) {$\mathcal{B}(\Sigma_n)$-cons.~over $\PA$};
\node (C) at (0, 7.2) {$\Sigma_n \land \Pi_n$-cons.~over $\PA$};

\node (Pi) at (0, 6) {$\Pi_n$-cons.~over $\PA$};
\node (Sigma) at (4, 5) {$\Sigma_n$-cons.~over $\PA$};

\node (SASn) at (-4, 4.8) {$\Sigma_n$-$(\SP, \Sigma_n)$};

\node (BAS) at (-4, 3.6) {$\mathcal{B}(\Sigma_n)$-$(\SP, \Sigma_n)$};
\node (DAB) at (0, 4.2) {$\Delta_{n+1}(\PA)$-$(\SP, \mathcal{B}(\Sigma_n))$};

\node (DAS) at (-4, 2.4) {$\Delta_{n+1}(\PA)$-$(\SP, \Sigma_n)$};
\node (SAB) at (0, 2.7) {$\Sigma_{n+1}$-$(\SP, \mathcal{B}(\Sigma_n))$};
\node (PAB) at (4, 2.7) {$\Pi_{n+1}$-$(\SP, \mathcal{B}(\Sigma_n))$};

\node (SAS) at (-4, 1.2) {$\Sigma_{n+1}$-$(\SP, \Sigma_n)$};
\node (PAS) at (0, 1.2) {$\Pi_{n+1}$-$(\SP, \Sigma_n)$};
\node (DDB) at (4, 1.2) {$\Delta_{n+1}(\PA)$-$(\Delta_n, \mathcal{B}(\Sigma_n))^*$};

\node (A) at (0,0) {$\SP$-cons.~over $\PA$};

\draw [->, double] (0, 9)--(B);
\draw [->, double] (B)--(C);
\draw [->, double] (C)--(Pi);
\draw [->, double] (C)--(Sigma);
\draw [->, double] (Sigma)--(SAB);
\draw [->, double] (Sigma)--(PAB);
\draw [->, double] (Pi)--(SASn);
\draw [->, double] (Pi)--(DAB);

\draw [->, double] (SASn)--(BAS);

\draw [->, double] (BAS)--(DAS);
\draw [->, double] (DAB)--(DAS);
\draw [->, double] (DAB)--(SAB);
\draw [->, double] (DAB)--(PAB);
\draw [->, double] (DAB)--(DDB);

\draw [->, double] (DAS)--(SAS);
\draw [->, double] (DAS)--(PAS);
\draw [->, double] (SAB)--(SAS);
\draw [->, double] (PAB)--(PAS);

\draw [->, double] (SAS)--(A);
\draw [->, double] (PAS)--(A);
\draw [->, double] (DDB)--(A);

\draw [->, double] (Sigma) to [out=0,in=0] (DDB);
\node at (5.8,4) {($n=1$)};

\end{tikzpicture}
\caption{Implications between the conditions}\label{Fig2}
\end{figure}

In the following, we show several non-implications. 

\begin{prop}\label{Prop_easy}\leavevmode
\begin{enumerate}
    \item There exists a theory $T$ such that $T$ is $\SP$-conservative over $\PA$ but is neither $\Pi_n$-conservative over $\PA$ nor $\Sigma_n$-conservative over $\PA$. 

    \item There exists a theory $T$ such that $T$ is $\Pi_n$-conservative over $\PA$ but is not $\Sigma_n$-conservative over $\PA$. 

    \item There exists a theory $T$ such that $T$ is $\Sigma_n$-conservative over $\PA$ but is not $\Pi_n$-conservative over $\PA$. 

    \item There exists a theory $T$ such that $T$ is $\Sigma_n \land \Pi_n$-conservative over $\PA$ but is not $\mathcal{B}(\Sigma_n)$-conservative over $\PA$. 
\end{enumerate}
\end{prop}
\begin{proof}
1. The authors proved \cite[Corollary 9.3]{KK} that there exists a $\Sigma_n \land \Pi_n$ sentence $\varphi$ such that $\varphi \in \Cons(\SP, \PA) \setminus \bigl(\Cons(\Pi_n, \PA) \cup \Cons(\Sigma_n, \PA) \bigr)$. 
Then, $T : = \PA + \varphi$ satisfies the required conditions. 

\medskip

Items 2 and 3 follow from Guaspari's theorem \cite[Theorem 2.4]{Gua} stating that $\Sigma_n \cap \Cons(\Pi_n, \PA) \setminus \Cons(\Sigma_n, \PA) \neq \emptyset$ and $\Pi_n \cap \Cons(\Sigma_n, \PA) \setminus \Cons(\Pi_n, \PA) \neq \emptyset$. 

\medskip

4. This item follows from authors' result \cite[Corollary 9.3]{KK} stating $\mathcal{B}(\Sigma_n) \cap  \Cons(\Sigma_n \wedge \Pi_n, \PA) \setminus \Cons(\mathcal{B}(\Sigma_n), \PA) \neq \emptyset$. 
\end{proof}

The following theorem is a refinement of the third clause of Proposition \ref{Prop_easy}. 

\begin{thm}\label{CE_1}
There exists a theory $T$ satisfying the following conditions: 
\begin{enumerate}
    \item $T$ is a $\Sigma_n$-$(\SP, \Sigma_n)$ theory. 

    \item $T$ is a $\Delta_{n+1}(\PA)$-$(\SP, \mathcal{B}(\Sigma_n))$ theory. 

    \item $T$ is $\Sigma_n$-conservative over $\PA$. 

    \item $T$ is not $\Pi_n$-conservative over $\PA$. 
\end{enumerate}
\end{thm}
\begin{proof}
By Fact \ref{Fact}, there exists a $\Sigma_n$ sentence $\rho$ such that $\rho \in \DCons(\Pi_n, \Sigma_n; T)$.
Let $\psi$ be a $\Pi_n$ sentence and $\varphi$ be a $\Delta_{n+1}(\PA)$ sentence satisfying the following equivalences: 
\[
\PA \vdash \psi \leftrightarrow \neg \Bigl(\PR_\PA (\gn{\psi}) \preccurlyeq \bigl(\PR_\PA^{\Pi_n}(\gn{\neg \varphi}) \vee \PR_{\PA + \rho}^{\Pi_n}(\gn{\neg \psi}) \bigr)\Bigr),
\]
\[
\PA \vdash \varphi \leftrightarrow \PR_{\PA}^{\Sigma_n}(\gn{\neg \varphi}) \preccurlyeq \PR_{\PA+ \psi}^{\Sigma_n \wedge \Pi_n}(\gn{\varphi}).
\]

In the following, we verify that the theory $T : = \PA + \psi$ satisfies the required conditions by proving several Claims.

\begin{cl} \label{cl1}
$\PA+ \varphi$ is $\Pi_n$-conservative over $\PA+ \psi$.
\end{cl}
\begin{proof}
Let $\pi$ be a $\Pi_n$ sentence and suppose $\PA+ \varphi \vdash \pi$. 
Since $\PA+ \neg \pi \vdash \neg \varphi$, there exists a natural number $p \in \omega$ such that $\PA + \neg \pi \vdash \Prf_{\PA}^{\Sigma_n}(\gn{\neg \varphi}, \num{p})$.
Since $\PA + \psi + \neg \varphi \vdash \forall z < \num{p} \, \neg \Prf_{\PA + \psi}^{\Sigma_n \wedge \Pi_n}(\gn{\varphi},z)$, we obtain $\PA + \psi + \neg \varphi + \neg \pi \vdash \varphi$, and hence $\PA + \psi + \neg \pi \vdash \varphi$. 
Since $\PA + \varphi \vdash \pi$, it follows that $\PA + \psi \vdash \pi$.
\end{proof}

\begin{cl}\label{cl2}
$\neg \varphi \in \HCons(\mathcal{B}(\Sigma_n), \PA+ \psi)$.
\end{cl}
\begin{proof}
Let $U$ be any subthery of $\PA + \psi$.
Let $\sigma \in \Sigma_n$ and $\pi \in \Pi_n$ be such that $U + \neg \varphi \vdash \sigma \vee \pi$. 
Since $U + \neg \sigma + \neg \pi \vdash \varphi$, there exists a natural number $p \in \omega$ such that $\PA + \neg \sigma + \neg \pi \vdash \Prf_{\PA + \psi}^{\Sigma_n \wedge \Pi_n}(\gn{\varphi}, \num{p})$. 
Since $\PA + \varphi \vdash \forall z \leq \num{p} \, \neg \Prf_{\PA}^{\Sigma_n}(\gn{\neg \varphi})$, we obtain $\PA + \neg \sigma + \neg \pi + \varphi \vdash \neg \varphi$.
Thus, $\PA+ \neg \sigma + \neg \pi \vdash \neg \varphi$ holds.
We conclude $U \vdash \sigma \vee \pi$.
\end{proof}

\begin{cl}\label{cl3}
$\PA \nvdash \psi$.
\end{cl}
\begin{proof}
Suppose, towards a contradiction, that $\PA \vdash \psi$. Then, there exists a natural number $p \in \omega$ such that $\PA \vdash \Prf_{\PA}(\gn{\psi}, \num{p})$.
Since $\PA + \rho +\varphi + \psi \vdash \forall z \leq \num{p} \ \neg \Prf_{\PA}^{\Pi_n}(\gn{\neg \varphi},z) \wedge \neg \Prf_{\PA+ \rho}^{\Pi_n}(\gn{\neg \psi},z)$, we obtain $\PA + \rho + \varphi + \psi \vdash \neg \psi$, which implies $\PA + \rho + \varphi \vdash \neg \psi$.
Since $\PA \vdash \psi$, we have $\PA + \varphi \vdash \neg \rho$.
By Claim \ref{cl1}, we obtain $\PA + \psi \vdash \neg \rho$. 
Then, $\PA \vdash \neg \rho$ because $\PA \vdash \psi$. 
Since $\rho \in \Cons(\Pi_n,\PA)$, this contradicts the consistency of $\PA$.
\end{proof}

Therefore, $\PA + \psi$ is not $\Pi_n$-conservative over $\PA$. 

\begin{cl}\label{cl4}
$\PA+ \varphi$ is $\Sigma_n$-conservative over $\PA + \neg \psi$.
\end{cl}
\begin{proof}
Let $\sigma \in \Sigma_n$ be such that $\PA + \varphi \vdash \sigma$.
Since $\PA + \neg \sigma \vdash \neg \varphi$, there exists a natural number $p \in \omega$ such that $\PA + \neg \sigma \vdash \Prf_{\PA}^{\Pi_n}(\gn{\neg \varphi}, \num{p})$.
It follows from Claim \ref{cl3} that $\PA + \neg \sigma \vdash \psi$, and equivalently, $\PA + \neg \psi \vdash \sigma$.
\end{proof}

\begin{cl}\label{cl5}
$\psi \in \HCons(\Sigma_n, \PA+ \rho)$.
\end{cl}
\begin{proof}
Let $U$ be any subtheory of $\PA + \rho$ and $\sigma \in \Sigma_n$ be such that $U + \psi \vdash \sigma$. 
Then, there exists a natural number $p \in \omega$ such that $\PA + \neg \sigma \vdash \Prf_{\PA + \rho}(\gn{\neg \psi}, \num{p})$. By Claim \ref{cl3}, we obtain $\PA + \neg \sigma \vdash  \psi$.
Thus, $U \vdash \sigma$.
\end{proof}

Thus, in particular, $\PA + \psi$ is $\Sigma_n$-conservative over $\PA$. 

\begin{cl}\label{cl6}
$\PA + \psi$ is a $\Sigma_n$-$(\SP, \Sigma_n)$ theory. 
\end{cl}
\begin{proof}
We prove that the $\Sigma_n$ sentence $\rho$ is a witness of Claim \ref{cl6}. 
By Claim \ref{cl5} and $\rho \in \Cons(\Pi_n,\PA)$, it is proved that $\PA + \psi + \rho$ is $\SP$-conservative over $\PA$ as in the proof of the implication $(1 \Rightarrow 2)$ in Theorem \ref{H*_Theta_Gamma_B}. 

We prove $\neg \rho \in \HCons(\Sigma_n, \PA + \psi)$.
Let $U$ be any subtheory of $\PA + \psi$ and $\sigma \in \Sigma_n$ be such that $U + \neg \rho \vdash \sigma$. 
Then, $\PA + \psi + \neg \rho \vdash \sigma$, and so $\PA + \psi \vdash \rho \vee \sigma$. 
Since $\rho \vee \sigma \in \Sigma_n$, by Claim \ref{cl5}, we obtain $\PA \vdash \rho \vee \sigma$.
Since $\PA + \neg \rho \vdash \sigma$ and $\neg \rho \in \Cons(\Sigma_n,\PA)$, it follows that $\PA \vdash \sigma$, and hence $U \vdash \sigma$.
\end{proof}

\begin{cl}\label{cl7}
$\PA + \psi$ is a $\Delta_{n+1}(\PA)$-$(\SP, \mathcal{B}(\Sigma_n))$ theory. 
\end{cl}
\begin{proof}
We prove that the $\Delta_{n+1}(\PA)$ sentence $\varphi$ is a witness of Claim \ref{cl7}.
At first, we prove $\PA + \varphi \vdash \psi$.
By Claim \ref{cl2} and $\PA + (\psi \vee \varphi) + \neg \varphi \vdash \psi$, we obtain $\PA + (\psi \vee \varphi) \vdash \psi$ because $\PA + (\psi \lor \varphi)$ is a subtheory of $\PA + \psi$. 
It follows that $\PA + \varphi \vdash \psi$.

Since $\neg \varphi \in \HCons(\mathcal{B}(\Sigma_n),\PA + \psi)$ holds by Claim \ref{cl2}, it suffices to prove thta $\PA + \psi + \varphi$ is $\SP$-conservative over $\PA$.
Let $\sigma \in \Sigma_n$ and $\pi \in \Pi_n$ be such that $\PA + \psi + \varphi \vdash \sigma \wedge \pi$. 
Since $\PA + \varphi \vdash \psi$, we have $\PA + \varphi \vdash \sigma \wedge \pi$. 
By Claims \ref{cl1} and \ref{cl4}, it follows that $\PA + \psi \vdash \pi$ and $\PA + \neg \psi \vdash \sigma$. 
Thus, we obtain $\PA \vdash \sigma \vee \pi$.
\end{proof}

We have completed our proof of Theorem \ref{CE_1}.
\end{proof}

We close this subsection by examining the $n=1$ case. 

\begin{prop}\label{CE_2}
For any theory $T$ with $T \vdash \PA + \Con(\PA)$, the following are equivalent: 
\begin{enumerate}
    \item $T$ is $\Sigma_1{\downarrow}\Pi_1$-conservative over $\PA$. 

    \item $T$ is $\Sigma_1$-conservative over $\PA$. 
\end{enumerate}
\end{prop}
\begin{proof}
It suffices to show the implication $(1 \Rightarrow 2)$ holds. 
Suppose that $T \vdash \PA + \Con(\PA)$ is $\Sigma_1 {\downarrow} \Pi_1$-conservative over $\PA$. 
Let $\sigma \in \Sigma_1$ be such that $T \vdash \sigma$. 
Let $\psi$ be a $\Sigma_1$ sentence satisfying the following equivalence:
\[
    \PA \vdash \psi \leftrightarrow \PR_{\PA}(\gn{\neg \psi}) \preccurlyeq \sigma. 
\]
Let $\psi^\bot$ be the $\Sigma_1$ sentence $\sigma \prec \PR_{\PA}(\gn{\neg \psi})$. 
Then, $\PA$ proves $(\sigma \lor \neg \Con(\PA) )\leftrightarrow \PR_{\PA}(\gn{\neg \psi})$ and $\sigma \to \psi \lor \psi^\bot$ (cf.~the FGH Theorem \cite{Kur23}).
Since $\PA + \Con(\PA) \vdash \RFN_{\Pi_1}(\PA)$, we have $\PA + \Con(\PA) \vdash \PR_{\PA}(\gn{\neg \psi}) \to \neg \psi$. 
By combining these things, we obtain $T \vdash \psi^\bot \land \neg \psi$. 
By the $\Sigma_1{\downarrow}\Pi_1$-conservativity of $T$ over $\PA$, we have $\PA \vdash \psi^\bot \lor \neg \psi$. 
Since $\PA \vdash \psi^\bot \to \neg \psi$, we get $\PA \vdash \neg \psi$ and thus $\PA \vdash \PR_{\PA}(\gn{\neg \psi})$. 
We have $\PA + \Con(\PA) \vdash \sigma$. 
Since $\Con(\PA)$ is $\Sigma_1$-conservative over $\PA$, we conclude $\PA \vdash \sigma$. 
\end{proof}

\begin{prop}\label{CE_3}
Every theory $T$ containing $\PA + \Con(\PA)$ is not a $\Delta_2(\PA)$-$(\Sigma_1{\downarrow}\Pi_1, \Sigma_1)$ theory. 
\end{prop}
\begin{proof}
    Suppose, towards a contradiction, that $T \vdash \PA + \Con(\PA)$ and $T$ is a $\Delta_2(\PA)$-$(\Sigma_1{\downarrow}\Pi_1, \Sigma_1)$ theory. 
    Then, we find a $\Delta_2(\PA)$ sentence $\varphi$ such that $T + \varphi$ is $\Sigma_1{\downarrow}\Pi_1$-conservative over $\PA$ and $\neg \varphi \in \HCons(\Sigma_1, T)$. 
    Since $T + \varphi$ is an extension of $\PA + \Con(\PA)$, by Proposition \ref{CE_2}, $T + \varphi$ is $\Sigma_1$-conservative over $\PA$. 
    In particular, $\varphi$ is in $\Delta_2(\PA) \cap \HDCons(\Sigma_1, \Sigma_1; T)$. 
    This contradicts Proposition \ref{HDCons_S_S}. 
\end{proof}

The following corollary is a refinement of the second clause of Proposition \ref{Prop_easy} in the $n=1$ case. 

\begin{cor}\label{Cor_1}
There exists a theory $T$ satisfying the following conditions: 
\begin{enumerate}
    \item $T$ is $\Sigma_1$-conservative over $\PA$. 

    \item $T$ is not a $\Delta_2(\PA)$-$(\Sigma_1{\downarrow}\Pi_1, \Sigma_1)$ theory. 
\end{enumerate}
\end{cor}
\begin{proof}
Let $T: = \PA + \Con(\PA)$. 
Since $T$ is $\Sigma_1$-sound, it is shown that $T$ is $\Sigma_1$-conservative over $\PA$. 
By Proposition \ref{CE_3}, $T$ is not a $\Delta_2(\PA)$-$(\Sigma_1{\downarrow}\Pi_1, \Sigma_1)$ theory. 
\end{proof}

We obtain a corollary stating that the converse implication of Proposition \ref{Prop_triple} is not the case in general.

\begin{cor}\label{Cor_triple}
    There exists a theory $T$ such that $\Sigma_1 \cap \HDCons(\Sigma_1 {\downarrow} \Pi_1, \Sigma_1; T) \neq \emptyset$ but $T$ is not a $\Sigma_1$-$(\Sigma_1 {\downarrow} \Pi_1, \Sigma_1)$ theory. 
\end{cor}
\begin{proof}
Let $T$ be a theory guaranteed by Corollary \ref{Cor_1}. 
Since $T$ is $\Sigma_1{\downarrow}\Pi_1$-conservative over $\PA$, by Fact \ref{DA0}, we have $\Sigma_1 \cap \HDCons(\Sigma_1 {\downarrow} \Pi_1, \Sigma_1; T) \neq \emptyset$. 
Since $\Sigma_1$-$(\Sigma_1{\downarrow}\Pi_1, \Sigma_1)$ implies $\Delta_2(\PA)$-$(\Sigma_1{\downarrow}\Pi_1, \Sigma_1)$, we have that $T$ is not a $\Sigma_1$-$(\Sigma_1{\downarrow}\Pi_1, \Sigma_1)$ theory. 
\end{proof}

Finally, we show that $\Delta_2(\PA)$-$(\Delta_1, \mathcal{B}(\Sigma_1))^*$ is strictly weaker than $\Delta_2(\PA)$-$(\Sigma_1{\downarrow}\Pi_1, \mathcal{B}(\Sigma_1))$ (cf.~Definition \ref{Def_truple_Delta}). 

\begin{cor}\label{Cor_triple2}
    There exists a theory $T$ such that $T$ is a $\Delta_2(\PA)$-$(\Delta_1, \mathcal{B}(\Sigma_1))^*$ theory but is not a $\Delta_2(\PA)$-$(\Sigma_1{\downarrow}\Pi_1, \mathcal{B}(\Sigma_1))$ theory. 
\end{cor}
\begin{proof}
Let $T$ be a theory guaranteed by Corollary \ref{Cor_1}. 
By Proposition \ref{S1-sound}, $T$ is a $\Delta_2(\PA)$-$(\Delta_1, \mathcal{B}(\Sigma_1))^*$ theory. 
Since $\Delta_2(\PA)$-$(\Sigma_1{\downarrow}\Pi_1, \mathcal{B}(\Sigma_1))$ implies $\Delta_2(\PA)$-$(\Sigma_1{\downarrow}\Pi_1, \Sigma_1)$, we conclude that $T$ is not a $\Delta_2(\PA)$-$(\Sigma_1{\downarrow}\Pi_1, \mathcal{B}(\Sigma_1))$ theory. 
\end{proof}

\section*{Acknowledgments}

The second author was supported by JSPS KAKENHI Grant Number JP23K03200.

\bibliographystyle{plain}
\bibliography{ref}

\end{document}